\RequirePackage{fix-cm}
\documentclass{svjour3}                     
\smartqed  
\usepackage{graphicx}
\usepackage{amssymb}
\usepackage{bm,amsmath,algorithm,algorithmic,url,booktabs}
\usepackage{comment}
\usepackage{cases}
\usepackage{multirow}
\usepackage{threeparttable}

\DeclareMathOperator{\subjectto}{subject~to}

\newcommand{\minimize}{\mathop{\rm minimize}\limits}
\newcommand{\maximize}{\mathop{\rm maximize}\limits}
\journalname{Springer Journal}
\begin{document}

\title{Bilevel Cutting-plane Algorithm for Solving Cardinality-constrained Mean-CVaR Portfolio Optimization Problems
}

\titlerunning{Cardinality-constrained Mean-CVaR Portfolio Optimization}        

\author{Ken Kobayashi \and  Yuichi Takano \and Kazuhide Nakata
}

\institute{Ken Kobayashi (corresponding author) \at
              Artificial Intelligence Laboratory, Fujitsu Laboratories Ltd., 4-1-1 Kamikodanaka, Nakahara-ku, Kawasaki-shi, Kanagawa 211-8588, Japan \\
              Tel.: +81-44-754-2328\\
              Fax: +81-44-754-2664\\
              \email{ken-kobayashi@fujitsu.com}           
           \and
           Yuichi Takano \at
              Faculty of Engineering, Information and Systems, University of Tsukuba, 1-1-1 Tennodai, Tsukuba-shi, Ibaraki 305-8573, Japan
           \and  
           Kazuhide Nakata \at
              School of Engineering, Tokyo Institute of Technology, 2-12-1 Ookayama, Meguro-ku, Tokyo 152-8552, Japan
}

\date{Received: date / Accepted: date}

\maketitle

\begin{abstract}
This paper studies mean-risk portfolio optimization models using the conditional value-at-risk (CVaR) as a risk measure. 
We also employ a cardinality constraint for limiting the number of invested assets.
Solving such a cardinality-constrained mean-CVaR model is computationally challenging for two main reasons. 
First, this model is formulated as a mixed-integer optimization (MIO) problem because of the cardinality constraint, so solving it exactly is very hard when the number of investable assets is large. 
Second, the problem size depends on the number of asset return scenarios, and the computational efficiency decreases when the number of scenarios is large. 
To overcome these challenges, we propose a high-performance algorithm named the \emph{bilevel cutting-plane algorithm} for exactly solving the cardinality-constrained mean-CVaR portfolio optimization problem.
We begin by reformulating the problem as a bilevel optimization problem and then develop a cutting-plane algorithm for solving the upper-level problem.
To speed up computations for cut generation, we apply to the lower-level problem another cutting-plane algorithm for efficiently minimizing CVaR with a large number of scenarios. 
Moreover, we prove the convergence properties of our bilevel cutting-plane algorithm. 
Numerical experiments demonstrate that, compared with other MIO approaches, our algorithm can provide optimal solutions to large problem instances faster.
\end{abstract}

\section{Introduction}\label{sec:intro}
Since the introduction of mean-variance portfolio analysis~\cite{Markowitz1952}, portfolio optimization models have been widely used in financial industries and actively studied by both academic researchers and institutional investors.
The traditional framework established by Markowitz~\cite{Markowitz1952} determines investment weights of financial assets with the aim of making low-risk high-return investments.
This paper focuses on mean-risk portfolio optimization models using the conditional value-at-risk (CVaR)~\cite{Rockafellar2000,Rockafellar2002} as a risk measure.

CVaR is a downside risk measure for evaluating a potential heavy loss. It is known as a coherent risk measure that has the desirable properties of monotonicity, translation invariance, positive homogeneity, and subadditivity \cite{Artzner1999,Pflug2000}. 
Additionally, CVaR is monotonic with respect to second-order stochastic dominance \cite{Pflug2000}, which means that CVaR minimization is consistent with the preference of any rational risk-averse decision-maker. 
In the standard scenario-based formulation~\cite{Rockafellar2000,Rockafellar2002}, many scenarios are required to approximate CVaR accurately~\cite{Kaut2007,Takeda2009}. 
To resolve this computational difficulty, various efficient algorithms have been proposed, including the dual solution method~\cite{Ogryczak2011}, nonsmooth optimization algorithms~\cite{Beliakov2006,Iyengar2013,Lim2010,Rockafellar2000}, the factor model~\cite{Konno2002}, cutting-plane algorithms~\cite{Ahmed2006,Haneveld2006,Kunzi-Bay2006,Takano2014}, the level method~\cite{Fabian2008}, smoothing methods~\cite{Alexander2006,Tong2010}, and successive regression approximations~\cite{Agoston2012}.  

We consider solving mean-CVaR portfolio optimization problems with a cardinality constraint for limiting the number of invested assets.
This constraint comes from real-world practice; if the number of invested assets is large, it is difficult for investors to monitor each asset, and substantial transaction costs are required~\cite{Mansini2014,Wilding2003}. 
Solving such a cardinality-constrained mean-CVaR model is computationally challenging for two main reasons. 
First, this model is formulated as a mixed-integer optimization (MIO) problem because of the cardinality constraint, so solving it exactly is very hard when the number of investable assets is large. 
Second, the problem size depends on the number of asset return scenarios. This decreases computational efficiency because sufficiently many scenarios are required to approximate CVaR accurately~\cite{Kaut2007,Takeda2009}.  
For this problem, some heuristic optimization algorithms have been developed, including continuous-relaxation-based heuristics~\cite{Angelelli2008}, the $\ell_1$-norm-based approximation of the cardinality constraint~\cite{Cheng2015}, and a Scholtes-type regularization method for complementarity constraints~\cite{Branda2018}. 
However, these algorithms cannot guarantee global optimality of obtained solutions. 
See Mansini \textit{et al.}~\cite{Mansini2014} for previous studies on cardinality-constrained portfolio optimization models. 

Recently, Bertsimas and Cory-Wright proposed a high-performance cutting-plane algorithm for exactly solving cardinality-constrained mean-variance portfolio optimization problems~\cite{Bertsimas2018}. 
They converted the mean-variance model into a bilevel optimization problem composed of an upper- and lower-level problem.  
This problem structure is fully exploited in their cutting-plane algorithm, where the upper-level problem is iteratively approximated using cutting planes derived from the dual formulation of the lower-level problem. 
Numerical experiments demonstrated that their cutting-plane algorithm was much faster than state-of-the-art MIO methods for solving large problem instances. 
More surprisingly, their algorithm succeeded in yielding optimal solutions to problem instances involving thousands of assets.

On the basis of this previous research, we propose a high-performance algorithm named the \textit{bilevel cutting-plane algorithm} for exactly solving the cardinality-constrained mean-CVaR portfolio optimization problem. 
For this purpose, we extend the cutting-plane algorithm~\cite{Bertsimas2018} that was originally developed for the mean-variance model. 
In the mean-CVaR model, however, the size of the lower-level problem depends on the number of scenarios, which degrades the computational performance of the algorithm. 
To speed up the computations of the lower-level problem, we apply the cutting-plane algorithm~\cite{Ahmed2006,Haneveld2006,Kunzi-Bay2006,Takano2014} that was developed for efficiently minimizing CVaR. 
As a result, these two types of cutting-plane algorithms are integrated into our bilevel cutting-plane algorithm to achieve faster computations.
We also prove its convergence properties. 
To our knowledge, we are the first to develop an effective algorithm for exactly solving cardinality-constrained mean-CVaR portfolio optimization problems.

We conducted numerical experiments using some benchmark datasets~\cite{Beasley1990,Kenneth} to evaluate the efficiency of our bilevel cutting-plane algorithm.
Numerical results demonstrate that our algorithm is faster than other MIO approaches, especially for large problem instances.
Remarkably, our algorithm attained an optimal solution within 3600~s to a problem involving 225 assets and 100,000 scenarios. 

The rest of the paper is organized as follows.
In Section~\ref{sec:formulation}, we formulate the cardinality-constrained mean-CVaR portfolio optimization problem.
In Section~\ref{sec:bcpa}, we describe our bilevel cutting-plane algorithm for solving the problem.
We report computational results in Section~\ref{sec:experiments} and conclude in Section~\ref{sec:conclusion}.

\section{Problem Formulation}\label{sec:formulation}
In this section, we formulate the cardinality-constrained mean-CVaR portfolio optimization problem that we consider in this paper. 

\subsection{Cardinality constraint}
Let $\bm x := (x_1, x_2,\ldots, x_N)^\top$ be a portfolio, where $x_n$ is the investment weight of the $n$th asset. 
Assume that
\begin{equation}\label{eq:weight}
\bm 1^\top \bm x = 1, \quad \bm x\geq \bm 0,
\end{equation}
where $\bm 1$ is the vector whose entries are all one.
The nonnegativity constraint on $\bm x$ prohibits short selling.

A cardinality constraint is employed to limit the number of assets held. 
Let $k$ be a user-defined parameter for limiting the cardinality. 
We then impose the following constraint on portfolio $\bm x$:
\begin{equation}\label{eq:cardinality}
  \|\bm x\|_0 \leq k,   
\end{equation}
where $\|\cdot\|_0$ is the $\ell_0$-norm (i.e., the number of nonzero entries).
In practice, investors require this constraint to reduce their portfolio monitoring and transaction costs.

Let $\bm z := (z_1,z_2,\ldots, z_N)^\top$ be a vector of binary decision variables for selecting assets; that is, $z_n = 1$ if the $n$th asset is selected, and $z_n=0$ otherwise. 
We also introduce the feasible set corresponding to the cardinality constraint~\eqref{eq:cardinality}: 
\begin{equation*}
\mathcal{Z}_N^k := \left\{\bm z\in \{0,1\}^N  ~\middle|~ \sum_{n\in \mathcal{N}} z_n\leq k\right\},
\end{equation*}
where $\mathcal{N}:= \{1,2,\ldots, N\}$ is the index set of $N$ assets.

Then, the cardinality constraint \eqref{eq:cardinality} is represented by the following logical implication:
\begin{numcases}
    {}
    z_n  = 0 ~\Rightarrow~ x_n = 0 &($\forall n\in \mathcal{N}$), \label{constr:logic}\\
    \bm z \in \mathcal{Z}_N^k.\label{constr:cardinality_z}
\end{numcases}

\subsection{Conditional value-at-risk}
Let $\tilde{\bm r} := (\tilde{r}_1, \tilde{r}_2,\ldots, \tilde{r}_N)^\top$ be a vector representing the rate of random return of each asset, and $\mathcal{P}: \mathbb{R}^N\rightarrow \mathbb{R}$ be the corresponding probability density function. 
The loss function is defined as the negative of the portfolio net return:
\begin{equation*}
    \mathcal{L}(\bm x, \tilde{\bm r}) := -\tilde{\bm r}^\top \bm x.
\end{equation*}

Let $\beta\in (0,1)$ be a probability level parameter, which is frequently set close to one.
Then, $\beta$-CVaR can be regarded as the approximate conditional expectation of a random loss exceeding the $\beta$-value-at-risk ($\beta$-VaR).
We use the following function to calculate CVaR: 
\begin{equation}\label{eq:CVaR_function}
    \mathcal{F}_{\beta}(a,\bm x) := a + \frac{1}{1-\beta}\int_{\bm r \in \mathbb{R}^N}[\mathcal{L}(\bm x, \bm r)-a]_+\mathcal{P}(\bm r)\,\textup{d} \bm r,
\end{equation}
where $a$ is a decision variable corresponding to $\beta$-VaR, and $[\xi]_+$ is a positive part of $\xi$ (i.e., $[\xi]_+ := \max\{0,\xi\}$).
Then, $\beta$-CVaR of the portfolio $\bm x$ is calculated as follows~\cite{Rockafellar2000,Rockafellar2002}:
\begin{equation*}
    \minimize_{a}~\mathcal{F}_{\beta}(a,\bm x).
\end{equation*}

Because multiple integration in Eq.~\eqref{eq:CVaR_function} is computationally expensive, the scenario-based approximation is commonly used. 
Let $\mathcal{S}:=\{1,2,\ldots,S\}$ be the index set of $S$ scenarios. 
We then have 
\begin{align*}
    \mathcal{F}_{\beta}(a,\bm x) &\approx a + \frac{1}{1-\beta}\sum_{s\in \mathcal{S}}p_s\left[-(\bm r^{(s)})^\top \bm x-a\right]_+,
\end{align*}
where $\bm r^{(s)} := (r^{(s)}_1,r^{(s)}_2,\ldots,r_N^{(s)})^\top$ is the $s$th scenario of asset returns generated from the probability density function $\mathcal{P}$, and $\bm p := (p_1,p_2,\ldots,p_S)^\top$ is a vector composed of occurrence probabilities of scenarios $s \in \mathcal{S}$.

\subsection{Portfolio optimization model}
Throughout the paper, we consider the following set of feasible portfolios:
\begin{equation*}
    \mathcal{X} := \{\bm x \in \mathbb{R}^N\mid \bm A\bm x\leq \bm b,~\bm 1^\top \bm x = 1,~\bm x\geq \bm 0\},
\end{equation*}
where $\bm A \in \mathbb{R}^{M\times N}$ and $\bm b\in \mathbb{R}^M$ are given.
It is supposed that the linear constraint $\bm A\bm x\leq \bm b$ contains the expected return constraint: 
\begin{equation}\label{eq:con_exp_return}
    \bm \mu^\top \bm x \geq \bar{\mu},
\end{equation}
where $\bm \mu :=(\mu_1,\mu_2,\ldots,\mu_N)^{\top}$ is a vector composed of expected returns of assets, and $\bar{\mu}$ is a parameter of the required return level.

In line with previous studies~\cite{Bertsimas2018,DeMiguel2009,Gotoh2013,Gotoh2011}, we incorporate the $\ell_2$-regularization term into the objective from a perspective of robust optimization. 
The cardinality-constrained mean-CVaR portfolio optimization model is formulated as follows:
\begin{subequations}\label{prob:org}
\begin{alignat}{3}
    &\minimize_{a, v, \bm x, \bm z} &&\quad \frac{1}{2\gamma }\bm x^\top \bm x + a + v \label{eq:org_obj}\\
    &\subjectto && \quad v \geq  \frac{1}{1-\beta}\sum_{s\in \mathcal{S}}p_s\left[- (\bm r^{(s)})^\top \bm x-a\right]_+, \label{eq:org_constr_cvar}\\
                &&&\quad z_n  = 0 ~\Rightarrow~ x_n = 0 \quad (\forall n\in \mathcal{N}), \label{constr:logic_z}\\
               &&&\quad  \bm x \in \mathcal{X}, \quad \bm z \in \mathcal{Z}_N^k, \label{constr:xz}
\end{alignat}
\end{subequations}
where $v$ is an auxiliary decision variable, and $\gamma > 0$ is a regularization parameter.

\section{Cutting-plane Algorithms}\label{sec:bcpa}
In this section, we present our bilevel cutting-plane algorithm for solving the cardinality-constrained mean-CVaR portfolio optimization problem~\eqref{prob:org}. 

\subsection{Bilevel optimization reformulation}
Here, we extend the method of bilevel optimization reformulation~\cite{Bertsimas2018} to the mean-CVaR portfolio optimization problem~\eqref{prob:org}. 
Let us define $\bm Z := \mathrm{Diag}(\bm z)$ as a diagonal matrix whose diagonal entries are given by $\bm z$. 
We first eliminate the logical implication~\eqref{constr:logic_z} by replacing $\bm x$ with $\bm Z\bm x$ in Problem~\eqref{prob:org}.
We next reformulate Problem (\ref{prob:org}) as a bilevel optimization problem. 
Specifically, the \emph{upper-level problem} is posed as the following integer optimization problem: 
\begin{equation}\label{prob:master_prob}
    \minimize_{\bm z}~f(\bm z) \quad \subjectto~\bm z \in \mathcal{Z}_N^k, 
\end{equation}
and its objective function is defined by the following \emph{lower-level problem}: 
\begin{subequations}\label{prob:sub}
\begin{alignat}{4}
    f(\bm z)=&\minimize_{a, v, \bm x} &&\quad \frac{1}{2\gamma }\bm x^\top \bm x + a + v \label{eq:sub_obj_primal}\\
    &\subjectto &&\quad v \geq  \frac{1}{1-\beta}\sum_{s\in \mathcal{S}}p_s\left[-(\bm r^{(s)})^\top \bm Z \bm x-a\right]_+, \label{eq:sub_constr_cvar}\\
                &&&\quad \bm Z \bm x \in \mathcal{X}.
\end{alignat}
\end{subequations} 
Note that $(\bm Z \bm x)^\top \bm Z \bm x$ has been replaced with $\bm x^\top \bm x$ in the objective~\eqref{eq:sub_obj_primal}; this is because $\bm x = \bm Z \bm x$ holds after minimization~\eqref{eq:sub_obj_primal} as in Bertsimas and Cory-Wright~\cite{Bertsimas2018}. 

We then derive cutting planes for the upper-level problem~\eqref{prob:master_prob} by exploiting the dual formulation of the lower-level problem~\eqref{prob:sub}. 
For this purpose, Constraint~\eqref{eq:sub_constr_cvar}, which is nonlinear and nondifferentiable, needs to be transformed into tractable constraints. 
A commonly used method is the lifting representation \cite{Fabian2008,Rockafellar2000}, which converts Problem~\eqref{prob:sub} into the following optimization problem with linear constraints:
\begin{subequations}\label{prob:sub2}
\begin{alignat}{4}
    f(\bm z)=&\minimize_{a, \bm q, v, \bm x} &&\quad \frac{1}{2\gamma }\bm x^\top \bm x + a + v \label{eq:sub2_obj}\\
    &\subjectto &&\quad v \geq \frac{1}{1-\beta}\sum_{s \in \mathcal{S}} p_sq_s,\label{eq:sub2_constr_cvar} \\
    &&&\quad q_s\geq - (\bm r^{(s)})^\top \bm Z\bm x - a & \quad (\forall s\in \mathcal{S}), \label{eq:sub2_constr_lift1}\\
    &&&\quad q_s \geq 0 & \quad (\forall s\in \mathcal{S}), \label{eq:sub2_constr_lift2} \\
    &&&\quad \bm Z \bm x \in \mathcal{X},
\end{alignat}
\end{subequations}
where $\bm q := (q_1,q_2,\ldots,q_S)^{\top}$ is a vector of auxiliary decision variables.
The following theorem gives the dual formulation of Problem~\eqref{prob:sub2}. 
\begin{theorem}\label{thm:f_dual_lifting}
Suppose that Problem~\eqref{prob:sub2} is feasible. 
Then, the strong duality holds, and the dual formulation of Problem~\eqref{prob:sub2} is represented as follows:
\begin{subequations}\label{prob:sub_dual_lift}
\begin{alignat}{3}
    f(\bm z) =&\maximize_{\bm \alpha, \bm \zeta, \lambda, \bm \omega} &&\quad -\frac{\gamma}{2} \bm z^\top  (\bm \omega \circ \bm \omega) - \bm b^\top \bm \zeta + \lambda\\
    &\subjectto &&\quad \bm \omega \geq \sum_{s\in S}\alpha_s \bm r^{(s)}-\bm A^\top\bm \zeta +\lambda \bm 1 ,\\
    &&&\quad \sum_{s\in S}\alpha_s = 1,\\
    &&&\quad \alpha_s  \leq \frac{p_s}{1-\beta} \quad (\forall s\in \mathcal{S}),\\
    &&&\quad \bm \alpha\geq \bm 0,\quad \bm \zeta \geq \bm 0,
\end{alignat}
where $\bm \alpha:=(\alpha_1,\alpha_2,\ldots,\alpha_S)^{\top},~\bm \zeta\in \mathbb{R}^M,~\lambda\in \mathbb{R}$, and $\bm \omega \in \mathbb{R}^N$ are dual decision variables, and $\bm \omega \circ \bm \omega$ denotes the Hadamard product of the vector $\bm \omega$.
\end{subequations}
\end{theorem}
\begin{proof}
See \ref{sec:proof_lift}.
\qed 
\end{proof}

\begin{remark}
Note that the dual problem~\eqref{prob:sub_dual_lift} is always feasible with $\bm \alpha = \bm p$, $\bm \zeta = \bm 0$, $\lambda = 0$, and $\bm \omega = \sum_{s\in \mathcal{S}}p_s\bm r^{(s)}$. 
From the strong duality, the dual problem~\eqref{prob:sub_dual_lift} is unbounded (i.e., $f(\bm z) = +\infty$) if and only if the corresponding primal problem~\eqref{prob:sub2} is infeasible. 
Our objective is to minimize $f(\bm z)$, so we can assume without loss of generality that the primal problem \eqref{prob:sub2} is feasible.
\end{remark}

As in Bertsimas and Cory-Wright~\cite{Bertsimas2018}, Theorem~\ref{thm:f_dual_lifting} shows two key properties of the function $f(\bm z)$. 

\begin{lemma}[Convexity]\label{lem:convex}
The function $f(\bm z)$ is convex in $\bm z\in [0,1]^N$. 
\end{lemma}
\begin{proof}
See, for example, Proposition~1.2.4 in Bertsekas \textit{et al.}~\cite{bertsekas2003convex}. 
\qed 
\end{proof}

\begin{lemma}[Subgradient]\label{lem:subgradient}
Suppose that Problem~\eqref{prob:sub2} is feasible, and that $\bm \omega^\star(\bm z)$ is an optimal solution of $\bm{\omega}$ to Problem~\eqref{prob:sub_dual_lift}. 
Then, a subgradient of the function $f(\bm z)$ is given by 
    \begin{equation}\label{eq:subgradient}
        \bm g (\bm z) := -\frac{\gamma}{2}\bm \omega^\star(\bm z)\circ\bm \omega^\star(\bm z) \in \partial f(\bm z). 
    \end{equation}
\end{lemma}
\begin{proof}
See, for example, Proposition~8.1.1 in Bertsekas \textit{et al.}~\cite{bertsekas2003convex}. 
\qed 
\end{proof}

Lemmas~\ref{lem:convex} and \ref{lem:subgradient} verify that for each $\hat{\bm z} \in \mathcal{Z}_N^k$, Problem~\eqref{prob:sub_dual_lift} with $\bm z = \hat{\bm z}$ yields a linear underestimator of $f(\bm z)$ for $\bm{z} \in [0,1]^N$ as follows: 
\begin{equation}\label{eq:cut}
    f(\bm z) \geq f(\hat{\bm z}) +  \bm g(\hat{\bm z})^\top (\bm z-\hat{\bm z}).
\end{equation}

\subsection{Upper-level cutting-plane algorithm}\label{sec:upper-level_cpa}
Here, we extend the cutting-plane algorithm~\cite{Bertsimas2018} to the upper-level problem~\eqref{prob:master_prob} for mean-CVaR portfolio optimization.
We refer to this algorithm as the \emph{upper-level cutting-plane algorithm}, which finds a sequence of solutions to relaxed versions of Problem~\eqref{prob:master_prob}.

Let $\theta_{\text{LB}}$ be a lower bound of the optimal objective value of Problem~\eqref{prob:master_prob}, which can be obtained, for example, by solving the associated continuous relaxation problem. 
We first define the initial feasible region as 
\begin{equation}\label{eq:initial_relax_region}
    \mathcal{F}_1 := \{(\bm z,\theta)\in  \mathcal{Z}_N^k\times \mathbb{R}\mid \theta \geq \theta_{\text{LB}}\},
\end{equation}
where $\theta$ is an auxiliary decision variable that serves as a lower bound on $f(\bm z)$.

At the $t$th iteration~($t\geq 1$), our algorithm solves the following optimization problem:
\begin{equation}\label{prob:master_relax}
    \minimize_{\bm z, \theta}~\theta\quad \subjectto~(\bm z,\theta) \in \mathcal{F}_t,
\end{equation}
where $\mathcal{F}_t$ is a relaxed feasible region at the $t$th iteration such that $\mathcal{F}_t \subseteq \mathcal{F}_1$.
According to Eq.~\eqref{eq:initial_relax_region}, the objective value of Problem \eqref{prob:master_relax} is bounded below. 
Therefore, unless $\mathcal{F}_t = \emptyset$, Problem \eqref{prob:master_relax} has an optimal solution, which is denoted by 
$(\bm z_t,\theta_t)$. 

After obtaining $(\bm z_t,\theta_t)$, we solve the dual lower-level problem \eqref{prob:sub_dual_lift} with $\bm z = \bm z_t$.
If Problem \eqref{prob:sub_dual_lift} is unbounded, Problem \eqref{prob:sub2} is infeasible because of the strong duality. 
In this case, we update the feasible region to cut off the solution $\bm z_t$ as follows:
\begin{equation}\label{eq:cut_infeas_sol}
    \mathcal{F}_{t+1} \leftarrow  \mathcal{F}_t \cap \{(\bm z,\theta)\in \mathcal{Z}_N^k\times \mathbb{R}\mid \bm z_t^\top (\bm 1- \bm z) + (\bm 1-\bm z_t)^\top \bm z\geq 1\}.
\end{equation}
If Problem \eqref{prob:sub_dual_lift} is feasible, we obtain the function value $f(\bm z_t)$ and its subgradient $\bm g(\bm z_t)$ as in Lemma \ref{lem:subgradient}. 

If $f(\bm z_t) - \theta_t \le \varepsilon$ with sufficiently small $\varepsilon \ge 0$, then $\bm z_t$ is an $\varepsilon$-optimal solution to Problem~\eqref{prob:master_prob}, which means that
\begin{equation}\label{eq:upper_optimality}
    f^\star \le f(\bm z_t) \le f^\star + \varepsilon,
\end{equation}
where $f^\star$ is the optimal objective value of Problem~\eqref{prob:master_prob}.
In this case, we terminate the algorithm with the $\varepsilon$-optimal solution $\bm z_t$.
Otherwise, we add the constraint~\eqref{eq:cut} to the feasible region: 
\begin{equation}\label{eq:cut_subgrad}
    \mathcal{F}_{t+1} \leftarrow \mathcal{F}_t \cap \{(\bm z,\theta)\in \mathcal{Z}_N^k\times \mathbb{R}\mid \theta \geq f(\bm z_t) +  \bm g(\bm z_t)^\top (\bm z - \bm z_t)\}.
\end{equation}
Note that this update cuts off the solution $(\bm z_t,\theta_t)$ because $\theta_t < f(\bm z_t)$.

After updating the feasible region, we set $t \leftarrow t + 1$ and solve Problem~\eqref{prob:master_relax} again. 
This procedure is repeated until an $\varepsilon$-optimal solution~$\hat{\bm z}$ is found.
After termination of the algorithm, we can compute the corresponding portfolio by solving Problem~\eqref{prob:sub2} with $\bm z = \hat{\bm z}$. 
Our upper-level cutting-plane algorithm is summarized by Algorithm \ref{alg:upper_level_cpa}.

 \begin{algorithm}[ht]
 \caption{Upper-level cutting-plane algorithm for solving Problem~\eqref{prob:master_prob}} 
 \label{alg:upper_level_cpa}
 \begin{algorithmic} 
 \normalsize
 \STATE \begin{description}
 \item[\textbf{Step 0~\textsf{(Initialization)}}] Let $\varepsilon \geq 0$ be a tolerance for optimality. 
 Define the feasible region~$\mathcal{F}_1$ as in Eq.~\eqref{eq:initial_relax_region}.
 Set $t\leftarrow 1$ and $\text{UB}_0 \leftarrow \infty$.
 \item[\textbf{Step 1 \textsf{(Relaxed Problem)}}] Solve Problem \eqref{prob:master_relax}. Let $(\bm z_t, \theta_t)$ be an optimal solution, and set $\text{LB}_t \leftarrow \theta_t$. 
\item[\textbf{Step 2 \textsf{(Cut Generation)}}] Solve Problem~\eqref{prob:sub_dual_lift} with $\bm z = \bm z_t$ to calculate $f(\bm z_t)$ and $\bm \omega^\star(\bm z_t)$. 
\renewcommand{\labelenumi}{(\alph{enumi})}
\renewcommand{\labelenumii}{\roman{enumii}.}
\begin{enumerate}
    \setcounter{enumi}{0}
    \item If Problem \eqref{prob:sub_dual_lift} is unbounded, set $\text{UB}_{t} \leftarrow \text{UB}_{t-1}$ and update the feasible region as in Eq.~\eqref{eq:cut_infeas_sol}.
    \item If Problem \eqref{prob:sub_dual_lift} is feasible, perform the following procedures: 
        \begin{enumerate}
        \setcounter{enumii}{0}
            \item Calculate $\bm g(\bm z_t)$ as in Eq.~\eqref{eq:subgradient}. 
            \item Update the feasible region as in Eq.~\eqref{eq:cut_subgrad}. 
            \item If $f(\bm z_t) < \text{UB}_{t-1}$, set $\text{UB}_{t} \leftarrow f(\bm z_t)$ and $(\hat{\bm z}, \hat{\theta}) \leftarrow (\bm z_t, \theta_t)$; otherwise, set $\text{UB}_{t} \leftarrow \text{UB}_{t-1}$. 
        \end{enumerate}
\end{enumerate}
\item[\textbf{Step 3 \textsf{(Termination Condition)}}] If $\text{UB}_t-\text{LB}_t\leq \varepsilon$, terminate the algorithm with the $\varepsilon$-optimal solution~$\hat{\bm z}$.
\item[\textbf{Step 4}] Set $t\leftarrow t+1$ and return to Step 1.
\end{description}
 \end{algorithmic} 
 \end{algorithm}
 
\subsection{Lower-level cutting-plane algorithm} \label{sec:lower-level_cpa}
Note that Step 2 of Algorithm \ref{alg:upper_level_cpa} solves Problem \eqref{prob:sub_dual_lift}, which contains $\mathcal{O}(S)$ decision variables and $\mathcal{O}(S)$ constraints; so, solving it directly is computationally expensive especially when the number of scenarios is very large. 
To avoid this difficulty, we propose solving the primal lower-level problem~\eqref{prob:sub} by instead using the other cutting-plane algorithm~\cite{Ahmed2006,Haneveld2006,Kunzi-Bay2006,Takano2014}. 
Hereinafter, we refer to this algorithm as the \emph{lower-level cutting-plane algorithm} to distinguish it from the upper-level cutting-plane algorithm (Algorithm \ref{alg:upper_level_cpa}). 
We will explain how to derive dual solutions for cut generation in Section~\ref{sec:efficient_calculation}. 

We first use the cutting-plane representation~\cite{Fabian2008,Kunzi-Bay2006} to reformulate Problem~\eqref{prob:sub} as follows:
\begin{subequations}\label{prob:sub3}
\begin{alignat}{4}
    f(\bm z)=&\minimize_{a, v, \bm x} &&\quad \frac{1}{2\gamma }\bm x^\top \bm x + a + v \label{eq:sub3_obj}\\
    &\subjectto &&\quad v \geq \frac{1}{1-\beta}\sum_{s \in \mathcal{J}}p_s(-(\bm r^{(s)})^\top \bm Z\bm x - a) \quad(\forall \mathcal{J}\subseteq \mathcal{S}),\label{eq:sub3_constr_cvar} \\
    &&&\quad \bm Z\bm x \in \mathcal{X}, \label{eq:sub3_constr_x}
\end{alignat}
\end{subequations}
where $\mathcal{J}$ denotes a subset of the scenario set $\mathcal{S}$.
We pick out the constraint~\eqref{eq:sub3_constr_cvar} with $\mathcal{J}=\mathcal{S}$ to define the initial feasible region:
\begin{equation}\label{eq:sub_primal_relaxed_region}
    \mathcal{G}_1(\bm z) := \left\{(a, v, \bm x)\in \mathbb{R} \times \mathbb{R} \times \mathbb{R}^N ~\middle|~\begin{aligned}&v \geq \frac{1}{1-\beta}\sum_{s\in \mathcal{S}}p_s(-(\bm r^{(s)})^\top \bm Z\bm x - a),\\ &v\geq 0,~\bm Z\bm x\in \mathcal{X}\end{aligned} \right\}.
\end{equation}
It is clear from Eq.~\eqref{eq:org_constr_cvar} that $v \geq 0$ is a valid constraint. 
At the $t$th iteration~($t\geq 1$), our algorithm solves the following optimization problem:
\begin{subequations}\label{prob:sub_primal_cutting_relaxed}
\begin{alignat}{4}
    &\minimize_{a, v, \bm x} &&\quad \frac{1}{2\gamma }\bm x^\top \bm x + a + v \\
    &\subjectto &&\quad (a, v, \bm x)\in \mathcal{G}_t(\bm z),
\end{alignat}
\end{subequations}
where $\mathcal{G}_t(\bm z)$ is a relaxed feasible region at the $t$th iteration such that $\mathcal{G}_t(\bm z) \subseteq \mathcal{G}_1(\bm z)$. 
The objective value of Problem \eqref{prob:sub_primal_cutting_relaxed} is bounded below by the initial feasible region~\eqref{eq:sub_primal_relaxed_region}. 
Therefore, unless $\mathcal{G}_t(\bm z) = \emptyset$, Problem \eqref{prob:sub_primal_cutting_relaxed} has an optimal solution, which is denoted by $(a_t, v_t, \bm x_t)$. 
Note also that Problem~\eqref{prob:sub_primal_cutting_relaxed} contains only $N+2$ decision variables, which are independent of the number of scenarios.

After obtaining $(a_t, v_t, \bm x_t)$, we define a subset of scenarios
\begin{equation}\label{eq:def_J}
    \mathcal{J}_t \leftarrow \{s\in \mathcal{S}\mid -(\bm r^{(s)})^\top \bm Z\bm x_t - a_t>0\}
\end{equation}
and calculate 
\begin{equation} \label{eq:calc_y'}
    v'_t \leftarrow \frac{1}{1-\beta}\sum_{s\in \mathcal{J}_t}p_s(-(\bm r^{(s)})^\top \bm Z\bm x_t - a_t).
\end{equation}
Because the revised solution~$(a_t, v'_t, \bm x_t)$ satisfies all the constraints of Problem~\eqref{prob:sub3}, this solution gives an upper bound of $f(\bm z)$. 

If $v'_t - v_t \leq \delta$ with sufficiently small $\delta \ge 0$, we terminate the algorithm with the $\delta$-optimal solution~$(a_t, v'_t, \bm x_t)$ to Problem \eqref{prob:sub3}. 
Otherwise, we update the feasible region as follows:
\begin{equation}\label{eq:cut_cvar_constraint}
    \mathcal{G}_{t+1}(\bm z)\leftarrow \mathcal{G}_t(\bm z) \cap \left\{(a, v, \bm x)\in \mathbb{R} \times \mathbb{R} \times \mathbb{R}^N ~\middle|~ v \geq \frac{1}{1-\beta}\sum_{s\in \mathcal{J}_t}p_s(-(\bm r^{(s)})^\top \bm Z\bm x-a)\right\}.
\end{equation}
Because $v_t < v'_t$, this update cuts off the solution $(a_t, v_t, \bm x_t)$.

Next, we set $t \leftarrow t + 1$ and solve Problem~\eqref{prob:sub_primal_cutting_relaxed} again. 
The lower-level cutting-plane algorithm is summarized by Algorithm~\ref{alg:kunzi-bay_cpa}.
As shown in previous studies~\cite{Ahmed2006,Haneveld2006,Kunzi-Bay2006,Takano2014}, Algorithm~\ref{alg:kunzi-bay_cpa} terminates in a finite number of iterations and outputs a $\delta$-optimal solution $(\hat{a},\hat{v},\hat{\bm x})$ to the lower-level problem~\eqref{prob:sub}. 
The resultant family of scenario subsets $\mathcal{K}_{\delta}(\bm z)$ is also provided by Algorithm~\ref{alg:kunzi-bay_cpa}. 
This algorithm is empirically much faster than solving Problem~\eqref{prob:sub2} based on the lifting representation when the number of scenarios is very large. 

Suppose that Algorithm~\ref{alg:kunzi-bay_cpa} terminates at the $T$th iteration and outputs a $\delta$-optimal solution $(\hat{a},\hat{v},\hat{\bm x})$, where $(a_T, v_T, \bm x_T)$ is an optimal solution to Problem \eqref{prob:sub_primal_cutting_relaxed} with $t=T$. 
We then obtain lower and upper bounds on $f(\bm z)$ as follows:
\begin{equation}\label{eq:ep_opt}
f(\bm z) - \delta \leq f_{\delta}(\bm z) \leq f(\bm z) \leq \hat{f}_{\delta}(\bm z) \leq f(\bm z) + \delta, 
\end{equation}
where 
\begin{align*}
    &f_\delta(\bm z):= \frac{1}{2\gamma }\bm x_T^\top \bm x_T + a_T + v_T, \quad \hat{f}_\delta(\bm z):=\frac{1}{2\gamma }\hat{\bm x}^\top \hat{\bm x} + \hat{a} + \hat{v}.   
\end{align*}

 \begin{algorithm}[ht]
 \caption{Lower-level cutting-plane algorithm for solving Problem~\eqref{prob:sub3}} 
 \label{alg:kunzi-bay_cpa}
 \begin{algorithmic} 
 \normalsize
 \STATE \begin{description}
 \item[\textbf{Step 0 \textsf{(Initialization)}}] Let $\delta\geq 0$ be a tolerance for optimality. 
 Define the feasible region $\mathcal{G}_1(\bm z)$ as in Eq.~\eqref{eq:sub_primal_relaxed_region}.
 \item[\textbf{Step 1 \textsf{(Relaxed Problem)}}] Solve Problem~\eqref{prob:sub_primal_cutting_relaxed}. 
 Let $(a_t, v_t, \bm x_t)$ be an optimal solution. 
 \item[\textbf{Step 2 \textsf{(Upper Bound)}}] 
 Calculate $\mathcal{J}_t$ and $v'_t$ as in Eqs.~\eqref{eq:def_J}~and~\eqref{eq:calc_y'}. 
 \item[\textbf{Step 3 \textsf{(Termination Condition)}}] If $v'_t - v_t \leq \delta$, terminate the algorithm with the $\delta$-optimal solution~$(\hat{a},\hat{v},\hat{\bm x}) := (a_t, v'_t, \bm x_t)$ and the resultant family of scenario subsets $\mathcal{K}_{\delta}(\bm z) := \{\mathcal{S},\mathcal{J}_1,\mathcal{J}_2,\ldots,\mathcal{J}_t\}$.
\item[\textbf{Step 4 \textsf{(Cut Generation)}}] Update the feasible region as in Eq.~\eqref{eq:cut_cvar_constraint}.
\item[\textbf{Step 5}] Set $t \leftarrow t+1$ and return to Step 1.
 \end{description}
 \end{algorithmic} 
 \end{algorithm}

\subsection{Efficient cut generation for the upper-level problem}\label{sec:efficient_calculation}
Now, we turn to the computation of dual solutions from primal solutions provided by the lower-level cutting-plane algorithm (Algorithm \ref{alg:kunzi-bay_cpa}). 
Such dual solutions are required for cut generation in the upper-level cutting-plane algorithm (Algorithm~\ref{alg:upper_level_cpa}). 

Suppose that Algorithm~\ref{alg:kunzi-bay_cpa} results in the following relaxed problem with $\mathcal{K}=\mathcal{K}_{\delta}(\bm z)$:
\begin{subequations}\label{prob:sub_primal_part}
    \begin{alignat}{4}
        f_{\mathcal{K}}(\bm z):=&\minimize_{a, v, \bm x} &&\quad \frac{1}{2\gamma }\bm x^\top \bm x + a + v\\
    &\subjectto &&\quad v \geq \frac{1}{1-\beta}\sum_{s \in \mathcal{J}}p_s( -(\bm r^{(s)})^\top \bm Z\bm x - a) &&\quad (\forall \mathcal{J} \in \mathcal{K}),\label{prob:sub_primal_part_const1}\\
    &&&\quad v \geq 0, \quad \bm Z\bm x \in \mathcal{X}. \label{prob:sub_primal_part_const2}
\end{alignat}
\end{subequations}
The following theorem derives the dual formulation of Problem~\eqref{prob:sub_primal_part}.  

\begin{theorem}\label{thm:f_dual_cutting}
Suppose that Problem~\eqref{prob:sub_primal_part} is feasible. 
Then, the strong duality holds, and the dual formulation of Problem \eqref{prob:sub_primal_part} can be written as
\begin{subequations} \label{prob:sub_dual_part}
\begin{alignat}{3}
    f_{\mathcal{K}}(\bm z)=&\maximize_{\bm \alpha, \bm \zeta, \lambda, \bm \omega} &&\quad -\frac{\gamma}{2} \bm z^\top (\bm \omega \circ \bm \omega) - \bm b^\top \bm \zeta + \lambda\label{prob:sub_dual_part_obj}\\
    &\subjectto && \quad \bm \omega \geq \sum_{\mathcal J \in \mathcal{K}}\alpha_{\mathcal{J}} \sum_{s\in \mathcal J}p_s \bm r^{(s)} -\bm A^\top \bm \zeta +\lambda \bm 1, \label{prob:sub_dual_part_const1}\\
    &&&\quad \sum_{\mathcal J \in \mathcal{K}} \alpha_{\mathcal{J}} \leq 1,\label{prob:sub_dual_part_const2}\\
    &&&\quad \sum_{\mathcal{J} \in \mathcal{K}}\alpha_{\mathcal{J}} \sum_{s\in \mathcal{J}}p_s = 1-\beta,\label{prob:sub_dual_part_const3}\\
    &&&\quad \bm \alpha \geq \bm 0, \quad \bm \zeta\geq \bm 0, \label{prob:sub_dual_part_const4}
\end{alignat}
\end{subequations}
where $\bm \alpha := (\alpha_{\mathcal{J}})_{\mathcal{J} \in \mathcal{K}}$,~$\bm \zeta\in \mathbb{R}^M,~\lambda\in \mathbb{R},$ and $\bm \omega \in \mathbb{R}^N$ are dual decision variables. 
\end{theorem}
\begin{proof}
See \ref{sec:proof_cutting}. 
\qed 
\end{proof}

Note that $|\mathcal{K}_{\delta}(\bm z)|$ is usually very small even when there are many scenarios. 
Therefore, Problem \eqref{prob:sub_dual_part} can be solved efficiently regardless of the number of scenarios. 

Let $\bm \omega^\star_{\delta}(\bm z)$ be the optimal solution of $\bm \omega$ for Problem \eqref{prob:sub_dual_part} with $\mathcal{K}=\mathcal{K}_{\delta}(\bm z)$. 
We then define
\begin{equation}\label{eq:subgradient2}
    \bm g_{\delta}(\bm z) := -\frac{\gamma}{2} \bm \omega^\star_{\delta}(\bm z)\circ \bm \omega^\star_{\delta}(\bm z).
\end{equation}
Note that when $\delta > 0$, $f_\delta(\bm z)$ and $\bm g_\delta(\bm z)$ are not exactly the same as $f(\bm z)$ and $\bm g(\bm z)$, respectively.
However, the following theorem verifies that $f_\delta(\bm z)$ and $\bm g_\delta(\bm z)$ still provide a linear underestimator of $f(\bm z)$.

\begin{theorem}
Suppose that Problem~\eqref{prob:sub_primal_part} is feasible with $(\bm z,\mathcal{K}) = (\hat{\bm z},\mathcal{K}_{\delta}(\hat{\bm z}))$. 
Then, it holds that for all $\bm z \in [0,1]^N$,
\begin{equation}\label{eq:thm5}
    f(\bm z) \geq f_{\delta}(\hat{\bm z}) + \bm g_{\delta}(\hat{\bm z})^\top(\bm z- \hat{\bm z}).
\end{equation}
\end{theorem}
\begin{proof}
Because $\mathcal{K}=\mathcal{K}_{\delta}(\hat{\bm z})$, it follows from Lemma \ref{lem:subgradient} that $\bm g_{\delta}(\hat{\bm z})$ is a subgradient of $f_{\mathcal{K}}(\bm z)$ at $\bm z = \hat{\bm z}$.
It then holds that for all $\bm z \in [0,1]^N$,
\begin{equation}\label{proof_eq2}
    f(\bm z) \geq f_{\mathcal{K}}(\bm z) 
    \geq f_{\mathcal{K}}(\hat{\bm z}) + \bm g_{\delta}(\hat{\bm z})^\top (\bm z-\hat{\bm z}).
\end{equation}
This proof is completed because $f_{\delta}(\hat{\bm z}) = f_{\mathcal{K}}(\hat{\bm z})$ with $\mathcal{K}=\mathcal{K}_{\delta}(\hat{\bm z})$.
\qed 
\end{proof}
 
\subsection{Bilevel cutting-plane algorithm}
We are now ready to describe our bilevel cutting-plane algorithm for solving the upper-level problem~\eqref{prob:master_prob}. 
We use the lower-level cutting-plane algorithm (Algorithm \ref{alg:kunzi-bay_cpa}) to accelerate the cut generation at Step 2 of Algorithm \ref{alg:upper_level_cpa}. 

Suppose that $(\bm z_t, \theta_t)$ is an optimal solution to the relaxed problem \eqref{prob:master_relax} at the $t$th iteration~$(t\geq 1)$. 
We then solve the primal lower-level problem \eqref{prob:sub3} with $\bm z=\bm z_t$ by means of the lower-level cutting-plane algorithm (Algorithm \ref{alg:kunzi-bay_cpa}), which provides $f_\delta(\bm z_t), \hat{f}_\delta(\bm z_t)$, and $\mathcal{K}_\delta(\bm z_t)$.
We next solve the reduced version~\eqref{prob:sub_dual_part} of the dual lower-level problem with $(\bm z, \mathcal{K}) = (\bm z_t, \mathcal{K}_{\delta}(\bm z_t))$, thereby yielding the optimal solution $\bm \omega_\delta^\star(\bm z_t)$. 
After that, we calculate $\bm g_\delta(\bm z_t)$ as in Eq.~\eqref{eq:subgradient2} and update the feasible region as follows:
\begin{equation}\label{eq:cut_subgrad2}
    \mathcal{F}_{t+1} \leftarrow \mathcal{F}_t \cap \{(\bm z,\theta)\in \mathcal{Z}_N^k\times \mathbb{R}\mid \theta \geq f_{\delta}(\bm z_t) +  \bm g_{\delta}(\bm z_t)^\top (\bm z - \bm z_t)\}.
\end{equation}

Our bilevel cutting-plane algorithm is summarized by Algorithm \ref{alg:bcpa}.
Note also that at Step 2 (Algorithm \ref{alg:bcpa}), portfolio $\hat{\bm x}$ is generated by Algorithm \ref{alg:kunzi-bay_cpa} at each iteration.

 \begin{algorithm}[ht]
 \caption{Bilevel cutting-plane algorithm for solving Problem~\eqref{prob:master_prob}} 
 \label{alg:bcpa}
 \begin{algorithmic}[0]
 \normalsize
 \STATE \begin{description}
 \item[\textbf{Step 0 \textsf{(Initialization)}}] Let $\varepsilon \geq 0$ and $\delta\geq 0$ be tolerances for optimality. 
 Define the feasible region $\mathcal{F}_1$ as in Eq.~\eqref{eq:initial_relax_region}.
 Set $t\leftarrow 1$ and $\text{UB}_0 \leftarrow \infty$.
\item[\textbf{Step 1 \textsf{(Relaxed Problem)}}] Solve Problem~\eqref{prob:master_relax}. 
Let $(\bm z_t, \theta_t)$ be an optimal solution, and set $\text{LB}_t \leftarrow \theta_t$. 
\item[\textbf{Step 2 \textsf{(Cut Generation)}}] Solve Problem~\eqref{prob:sub3} with $\bm z = \bm z_t$ by means of Algorithm \ref{alg:kunzi-bay_cpa} to calculate $f_\delta(\bm z_t), \hat{f}_\delta(\bm z_t),$ and $\mathcal{K}_{\delta}(\bm z_t)$. 
\renewcommand{\labelenumi}{(\alph{enumi})}
\renewcommand{\labelenumii}{\roman{enumii}.}
\begin{enumerate}
    \setcounter{enumi}{0}
    \item If Problem~\eqref{prob:sub3} is infeasible, set $\text{UB}_{t} \leftarrow \text{UB}_{t-1}$ and update the feasible region as in Eq.~\eqref{eq:cut_infeas_sol}.
    \item If Problem~\eqref{prob:sub3} is feasible, perform the following procedures:
        \begin{enumerate}
            \setcounter{enumii}{0}
            \item Solve Problem~\eqref{prob:sub_dual_part} with $(\bm z, \mathcal{K}) = (\bm z_t, \mathcal{K}_{\delta}(\bm z_t))$ to calculate $\bm \omega_\delta^\star(\bm z_t)$. 
            \item Calculate $\bm g_\delta(\bm z_t)$ as in Eq.~\eqref{eq:subgradient2}. 
            \item Update the feasible region as in Eq.~\eqref{eq:cut_subgrad2}.
            \item If $\hat{f}_\delta(\bm z_t) < \text{UB}_{t-1}$, set $\text{UB}_{t} \leftarrow \hat{f}_\delta(\bm z_t)$ and $(\hat{\bm z}, \hat{\theta}) \leftarrow (\bm z_t, \theta_t)$; otherwise, set $\text{UB}_{t} \leftarrow \text{UB}_{t-1}$. 
        \end{enumerate}
\end{enumerate}
\item[\textbf{Step 3 \textsf{(Termination Condition)}}] If $\bm z_u = \bm z_t$ for some $u < t$ or $\text{UB}_t-\text{LB}_t\leq \varepsilon$, then terminate the algorithm with the $\max\{\delta, \varepsilon\}$-optimal solution $\hat{\bm z}$.
\item[\textbf{Step 4}] Set $t\leftarrow t+1$ and return to Step 1. 
\end{description}
 \end{algorithmic} 
 \end{algorithm}
 
\begin{remark}
Several algorithms (e.g., primal-dual interior-point methods~\cite{wright1997primal} and alternating direction methods of multipliers~\cite{boyd2011distributed}) solve primal and dual problems simultaneously. 
By applying such algorithms in Step~1 of Algorithm \ref{alg:kunzi-bay_cpa}, we can skip solving the dual lower-level problem~\eqref{prob:sub_dual_part} at Step~2 of Algorithm~\ref{alg:bcpa}. 
\end{remark}

\begin{remark}
 Note that Step 1 of Algorithm \ref{alg:bcpa} solves the MIO problem~\eqref{prob:master_relax} at each iteration; this amounts to a ``multi-tree'' implementation, where a branch-and-bound algorithm is repeatedly executed from scratch. 
 To improve computational efficiency, we can use the function of \texttt{lazy constraint callback}, which is offered by modern optimization software (e.g., CPLEX or Gurobi). 
 This function enables a ``single-tree'' implementation~\cite{Quesada1992}; that is, cutting planes~\eqref{eq:cut_subgrad2} are dynamically generated during the process of the branch-and-bound algorithm. 
\end{remark}
 
To prove the convergence properties of Algorithm \ref{alg:bcpa}, we show the following lemma. 
\begin{lemma}\label{lem:cycle}
Suppose that $\{(\bm z_t,\theta_t) \mid t=1,2,\ldots, T \}$ is a sequence of solutions generated by Algorithm \ref{alg:bcpa}. 
If there exists $u < T$ such that $\bm z_u = \bm z_T$, then $\bm z_{T}$ is a $\delta$-optimal solution to Problem \eqref{prob:master_prob}, meaning that
\begin{equation*}
    f^\star \leq f(\bm z_T) \leq f^\star + \delta.
\end{equation*}
\end{lemma}
\begin{proof}
Note that $(\bm z_T,\theta_T)$ is contained in the feasible region~\eqref{eq:cut_subgrad2} at $t = u$.
Because $\bm z_u=\bm z_T$, it follows that
\[
\theta_T \geq f_{\delta}(\bm z_u) + \bm g_{\delta}(\bm z_u)^\top (\bm z_T - \bm z_u) = f_{\delta}(\bm z_u) = f_{\delta}(\bm z_{T}).
\]
From Eqs.~\eqref{eq:upper_optimality}~and~\eqref{eq:ep_opt}, we have 
\[
f^\star \leq f(\bm z_T), \quad f(\bm z_T) - \delta \leq f_{\delta}(\bm z_T). 
\]
Because $\theta_T \leq f^\star$, it follows that 
\begin{equation*}
    f^\star \leq f(\bm z_T) \leq f_{\delta}(\bm z_T) + \delta \leq \theta_T + \delta \leq f^\star + \delta,
\end{equation*}
which completes the proof. 
\qed 
\end{proof}

From Lemma \ref{lem:cycle}, we can establish the convergence properties of Algorithm \ref{alg:bcpa}.
\begin{theorem}
Algorithm \ref{alg:bcpa} terminates in a finite number of iterations and outputs a $max\{\delta, \varepsilon\}$-optimal solution.
\end{theorem}
\begin{proof}
Because there are at most a finite number of solutions $\bm z\in \mathcal{Z}_N^k$, Algorithm \ref{alg:bcpa} terminates in a finite number of iterations with $\bm z_u = \bm z_t$ for some $u < t$. 
In this case, a $\delta$-optimal solution is found according to Lemma \ref{lem:cycle}. 
If $\text{UB}_t-\text{LB}_t\leq \varepsilon$ is fulfilled first, Algorithm \ref{alg:bcpa} terminates with an $\varepsilon$-optimal solution. 
\qed 
\end{proof}

\section{Numerical Experiments}\label{sec:experiments}
In this section, we report numerical results to evaluate the efficiency of our algorithms for solving cardinality-constrained mean-CVaR portfolio optimization problems.  

\subsection{Problem instances}
Table \ref{tab:datalist} lists the datasets used in our experiments, where $N$ is the number of assets. 
From the data library on the website of Kenneth R.\ French~\cite{Kenneth}, we downloaded four historical datasets (i.e., \texttt{ind38}, \texttt{ind49}, \texttt{sbm25}, and \texttt{sbm100}) of US stock returns.  
We used monthly data from January 2010 to December 2019 to compute the mean vector $\bm \mu$ and covariance matrix $\bm \Sigma$ of asset returns.
From the OR-Library~\cite{Beasley1990}, we downloaded three datasets (i.e., \texttt{port1}, \texttt{port2}, and \texttt{port5}) of the mean vector $\bm \mu$ and covariance matrix $\bm \Sigma$ of asset returns, which were multiplied by 100 to be consistent with the other datasets.

\begin{table}[ht]
\normalsize
    \centering
    \caption{Dataset description}
    \label{tab:datalist}
    \begin{tabular}{lrl}
        \toprule
        Abbr. &$N$ &Original dataset  \\ \hline
        \texttt{ind38} &38 &38 Industry Portfolios \cite{Kenneth}\\
        \texttt{ind49} &49 &49 Industry Portfolios \cite{Kenneth}\\
        \texttt{sbm25} &25 & 
25 Portfolios Formed on Size and Book-to-Market \cite{Kenneth}\\
        \texttt{sbm100} &100 & 
100 Portfolios Formed on Size and Book-to-Market \cite{Kenneth}\\
        \texttt{port1} &31 &port1 (Portfolio optimization: Single period) \cite{Beasley1990}\\
        \texttt{port2} &89 &port2 (Portfolio optimization: Single period) \cite{Beasley1990}\\
        \texttt{port5} &225 &port5 (Portfolio optimization: Single period) \cite{Beasley1990}\\
        \bottomrule
    \end{tabular}
\end{table}

For each dataset, we randomly generated $S$ scenarios of asset returns $\{\bm r^{(s)} \mid s \in \mathcal{S}\}$ from $\mathrm{N}(\bm \mu, \bm \Sigma)$, a normal distribution with the parameters~$(\bm \mu, \bm \Sigma)$. 
We set the occurrence probability as $p_s = 1/S$ for all $s\in \mathcal{S}$.

The required return level in Eq.~\eqref{eq:con_exp_return} was set as
$\bar{\mu}= 0.3 \mu_{\text{min}} + 0.7 \mu_{\text{max}}$, where $\mu_{\text{max}}$ and $\mu_{\text{min}}$ were the average returns of the top- and bottom-$k$ assets, respectively. 
The probability level of CVaR was set as $\beta=0.9$.

\subsection{Methods for comparison}
A standard MIO formulation of Problem~\eqref{prob:org} uses the big-$M$ method, which replaces the logical implication~\eqref{constr:logic_z} with 
\begin{equation}\label{eq:bigM}
    0 \leq x_n \leq z_n\quad (\forall n\in \mathcal{N}),
\end{equation}
which is valid because of Eq.~\eqref{eq:weight}. 

Another state-of-the-art MIO formulation uses the perspective reformulation~\cite{Gunluk2010,Gunluk2012}:
\begin{subequations}\label{prob:persp}
\begin{alignat}{3}
    &\minimize_{a, v, \bm x, \bm y, \bm z} &&\quad \frac{1}{2\gamma }\sum_{n \in \mathcal{N}} y^2_n + a + v \label{eq:persp_obj}\\
    &\subjectto && \quad v \geq  \frac{1}{1-\beta}\sum_{s\in \mathcal{S}}p_s\left[-(\bm r^{(s)})^\top \bm x-a\right]_+, \label{eq:persp_constr_cvar}\\
                &&&\quad x_n \leq y_n z_n, \quad y_n \geq 0 \quad (\forall n\in \mathcal{N}), \label{constr:persp}\\
               &&&\quad  \bm x \in \mathcal{X}, \quad \bm z \in \mathcal{Z}_N^k, 
\end{alignat}
\end{subequations}
where $\bm y :=(y_1,y_2,\ldots,y_N)^{\top}$ is a vector of auxiliary decision variables.  

We compare the computational performances of the following methods:
\begin{description}
    \item[\textbf{BigM}] solves Problem~\eqref{prob:org} after replacing Eq.~\eqref{constr:logic_z} with Eq.~\eqref{eq:bigM};
    \item[\textbf{Persp}] solves Problem~\eqref{prob:persp} as a mixed-integer second-order cone optimization problem;
        \begin{description}
        \item[\textbf{Lift}] replaces Eqs.~\eqref{eq:org_constr_cvar}~and~\eqref{eq:persp_constr_cvar} with the lifting representation~\eqref{eq:sub2_constr_cvar}--\eqref{eq:sub2_constr_lift2};
        \item[\textbf{Cut}] applies the cutting-plane algorithm~\cite{Ahmed2006,Haneveld2006,Kunzi-Bay2006,Takano2014} to Problems \eqref{prob:org} and \eqref{prob:persp};
        \end{description}
    \item[\textbf{CP}] solves Problem~\eqref{prob:master_prob} by means of the upper-level cutting-plane algorithm (Algorithm \ref{alg:upper_level_cpa});
    \item[\textbf{BCP}] solves Problem \eqref{prob:master_prob} by means of the bilevel cutting-plane algorithm (Algorithm \ref{alg:bcpa});
    \item[\textbf{BCPc}] solves Problem \eqref{prob:master_prob} by means of the bilevel cutting-plane algorithm (Algorithm \ref{alg:bcpa}), where the callback function is employed in implementation.
\end{description}
All experiments were performed on a Windows 10 PC with an Intel Core i7-4790 CPU (3.6.0GHz) and 16 GB of memory.
All methods were implemented in Python 3.7 with Gurobi Optimizer 8.1.1. 
We employed the primal-dual interior-point method offered by Gurobi at Step~1 of Algorithm \ref{alg:kunzi-bay_cpa}. 
We also tested the performance of an alternating direction method of multipliers, but it made Algorithm \ref{alg:kunzi-bay_cpa} slower because of numerical inaccuracies.

We set $\varepsilon =\delta = 10^{-5}$ for tolerances for optimality in the cutting-plane algorithms. 
Except for BCP, we used \texttt{lazy constraint callback} for adding cutting planes during the branch-and-bound procedure.
The computation of each method was terminated if it did not finish by itself within 3600~s.
In these cases, the results obtained within 3600~s were taken as the final outcome.

\subsection{Evaluation metrics}
The row labels used in the tables of experimental results are defined as follows:
\begin{description}
\item[\textbf{Time}] computation time in seconds;
\item[\textbf{Obj}] objective value of the obtained best feasible solution;
\item[\textbf{Gap$(\%)$}] relative gap between lower and upper bounds on the optimal objective value;
\item[\textbf{\#Nodes}] number of nodes explored in the branch-and-bound algorithm; 
\item[\textbf{\#Cuts}] number of cutting planes generated.
\end{description}
Note that the best values of Time are indicated in bold for each instance, and those of Obj and Gap$(\%)$ are also indicated in bold for only the \texttt{port2} and \texttt{port5} datasets.

\subsection{Results for various numbers of scenarios}

First, we evaluate the computational performance of each method for various numbers of scenarios.
Here, we set the parameters $k=10$ for the cardinality constraint and $\gamma = 10/\sqrt{N}$ for the $\ell_2$-regularization term. 

Tables \ref{tbl:results_scenario_famafrench} and \ref{tbl:results_scenario_port} give the numerical results for each dataset with the number of scenarios $S \in \{10^3, 10^4, 10^5\}$.
Let us focus on our cutting-plane algorithms (i.e, CP, BCP, and BCPc). 
BCP and BCPc solved most of the problem instances faster than did CP, especially for large $S$.
In the case of the \texttt{ind49} dataset (Table \ref{tbl:results_scenario_famafrench}) for example, although CP was the fastest among the three methods for $S=10^3$, BCP and BCPc were much faster than CP for $S \geq 10^4$. 
These results suggest the effectiveness of the lower-level cutting-plane algorithm, which is used by BCP and BCPc to solve the lower-level problem efficiently regardless of the number of scenarios. 

BCP was faster than BCPc except in the case of the $\texttt{port5}$ dataset with $S=10^3$ (Table \ref{tbl:results_scenario_port}), but the differences in computation time between both the methods were small. 
For the \texttt{port2} dataset (Table \ref{tbl:results_scenario_port}), BCP and BCPc failed to complete the computations within 3600~s even when $S=10^3$; however, BCPc found solutions of better quality than did BCP. 
BCP provides at most one feasible solution at each iteration $t$, whereas BCPc explores feasible solutions more frequently through the callback procedure.  
For this reason, BCPc is capable of yielding good solutions even if the computation is terminated because of the time limit.

Next, we compare our cutting-plane algorithms with the MIO formulations (i.e., BigM and Persp). 
For medium-sized problem instances with $N< 100$ and $S \leq 10^4$, there was no clear inferior-to-superior relationship between those methods. 
When $S=10^5$, on the other hand, our algorithms BCP and BCPc often outperformed the other methods. 
Indeed, for the \texttt{sbm100} dataset with $S=10^5$ (Table \ref{tbl:results_scenario_famafrench}), BCP was much faster than the four methods related to BigM and Persp.
Moreover, for the \texttt{port5} dataset with $S=10^5$ (Table \ref{tbl:results_scenario_port}), only BCP and BCPc finished solving the problem within 3600~s, which is a remarkable result.
In the case of the \texttt{port5} dataset with $S=10^3$ (Table \ref{tbl:results_scenario_port}), Persp+Cut returned an incorrect optimal objective value because of numerical instability.

According to the convergence properties of the scenario-based approximation, a huge number of scenarios are required for calculating CVaR accurately~\cite{Takeda2009}. 
Even with only 15 investable assets, it was necessary to have at least 5000 scenarios to ensure the stability of the CVaR optimization model~\cite{Kaut2007}. 
These facts support the practicality of our bilevel cutting-plane algorithm, which has the potential to deal with many scenarios as shown in Tables \ref{tbl:results_scenario_famafrench} and \ref{tbl:results_scenario_port}.

\begin{table}[ht]
\scriptsize \renewcommand{\arraystretch}{0.9}
    \centering
    \caption{Numerical results for the datasets~\cite{Kenneth} with $(\gamma,k)=(10/\sqrt{N},10)$}
    \label{tbl:results_scenario_famafrench}
    \begin{tabular}{crrrrrrrrrr}
    \toprule
\multirow{2}{*}{Data} 	&\multirow{2}{*}{$N$}	&\multirow{2}{*}{$S$}	 &&\multicolumn{2}{c}{BigM}	 &\multicolumn{2}{c}{Persp}	 &\multirow{2}{*}{CP}	 &\multirow{2}{*}{BCP}	  &\multirow{2}{*}{BCPc}\\ \cmidrule(lr){5-6} \cmidrule(lr){7-8} 
        &       &        &&\multicolumn{1}{c}{Lift}&\multicolumn{1}{c}{Cut} &\multicolumn{1}{c}{Lift}&\multicolumn{1}{c}{Cut}\\ \midrule
\texttt{ind38}	 &38	 &$10^3$	 &Time	&1.3	&3.4	&1.7	&2.6	&2.4	&\textbf{0.9}	&1.4\\ 
	 &	 & 	 &Obj	&3.989	&3.989	&3.989	&3.989	&3.989	&3.989	&3.989\\ 
	 &	 & 	 &Gap(\%)	&0.00	&0.00	&0.00	&0.00	&0.00	&0.00	&0.00\\ 
	 &	 & 	 &\#Nodes	&1	&2693	&1	&3059	&19	& \multicolumn{1}{c}{---}	&0\\ 
	 &	 & 	 &\#Cuts	& \multicolumn{1}{c}{---}	&505	& \multicolumn{1}{c}{---}	&274	&20	&1	&7\\ \cline{3-11} 
 	 &	 &$10^4$	 &Time	&14.0	&15.9	&23.5	&6.7	&23.1	&\textbf{3.2}	&5.4\\ 
	 &	 & 	 &Obj	&3.820	&3.820	&3.820	&3.820	&3.820	&3.820	&3.820\\ 
	 &	 & 	 &Gap(\%)	&0.00	&0.00	&0.00	&0.00	&0.00	&0.00	&0.00\\ 
	 &	 & 	 &\#Nodes	&1	&3362	&1	&1597	&11	& \multicolumn{1}{c}{---}	&0\\ 
	 &	 & 	 &\#Cuts	& \multicolumn{1}{c}{---}	&479	& \multicolumn{1}{c}{---}	&178	&20	&1	&7\\ \cline{3-11} 
 	 &	 &$10^5$	 &Time	&493.9	&98.4	&3611.5	&57.8	&254.3	&\textbf{25.8}	&47.6\\ 
	 &	 & 	 &Obj	&3.860	&3.860	&3.860	&3.860	&3.860	&3.860	&3.860\\ 
	 &	 & 	 &Gap(\%)	&0.00	&0.00	&0.00	&0.00	&0.00	&0.00	&0.00\\ 
	 &	 & 	 &\#Nodes	&1	&2327	&1	&1139	&21	& \multicolumn{1}{c}{---}	&0\\ 
	 &	 & 	 &\#Cuts	& \multicolumn{1}{c}{---}	&338	& \multicolumn{1}{c}{---}	&193	&22	&1	&7\\ \hline 
 \texttt{ind49}	 &49	 &$10^3$	 &Time	&\textbf{1.8}	&4.8	&2.3	&11.8	&6.1	&7.3	&8.6\\ 
	 &	 & 	 &Obj	&3.095	&3.095	&3.095	&3.095	&3.095	&3.095	&3.095\\ 
	 &	 & 	 &Gap(\%)	&0.00	&0.00	&0.00	&0.00	&0.00	&0.00	&0.00\\ 
	 &	 & 	 &\#Nodes	&16	&3792	&23	&4506	&211	& \multicolumn{1}{c}{---}	&33\\ 
	 &	 & 	 &\#Cuts	& \multicolumn{1}{c}{---}	&573	& \multicolumn{1}{c}{---}	&1101	&76	&40	&45\\ \cline{3-11} 
 	 &	 &$10^4$	 &Time	&18.0	&21.2	&26.0	&26.2	&47.1	&\textbf{5.2}	&8.3\\ 
	 &	 & 	 &Obj	&3.343	&3.343	&3.343	&3.343	&3.343	&3.343	&3.343\\ 
	 &	 & 	 &Gap(\%)	&0.00	&0.00	&0.00	&0.00	&0.00	&0.00	&0.00\\ 
	 &	 & 	 &\#Nodes	&1	&3357	&1	&3815	&89	& \multicolumn{1}{c}{---}	&0\\ 
	 &	 & 	 &\#Cuts	& \multicolumn{1}{c}{---}	&515	& \multicolumn{1}{c}{---}	&580	&47	&1	&7\\ \cline{3-11} 
 	 &	 &$10^5$	 &Time	&709.8	&125.0	&$>$3600	&220.0	&662.2	&\textbf{43.1}	&103.6\\ 
	 &	 & 	 &Obj	&3.379	&3.379	&3.380	&3.379	&3.379	&3.379	&3.379\\ 
	 &	 & 	 &Gap(\%)	&0.00	&0.00	&0.00	&0.00	&0.00	&0.00	&0.00\\ 
	 &	 & 	 &\#Nodes	&31	&4101	&$>$1	&3489	&250	& \multicolumn{1}{c}{---}	&0\\ 
	 &	 & 	 &\#Cuts	& \multicolumn{1}{c}{---}	&380	& \multicolumn{1}{c}{---}	&661	&68	&1	&12\\ \hline 
 \texttt{sbm25}	 &25	 &$10^3$	 &Time	&1.3	&\textbf{0.1}	&1.7	&0.2	&1.3	&0.2	&0.7\\ 
	 &	 & 	 &Obj	&5.017	&5.017	&5.017	&5.017	&5.017	&5.017	&5.017\\ 
	 &	 & 	 &Gap(\%)	&0.00	&0.00	&0.00	&0.00	&0.00	&0.00	&0.00\\ 
	 &	 & 	 &\#Nodes	&1	&0	&1	&129	&0	& \multicolumn{1}{c}{---}	&0\\ 
	 &	 & 	 &\#Cuts	& \multicolumn{1}{c}{---}	&17	& \multicolumn{1}{c}{---}	&14	&4	&1	&8\\ \cline{3-11} 
 	 &	 &$10^4$	 &Time	&11.8	&0.5	&21.5	&\textbf{0.5}	&11.3	&1.1	&2.7\\ 
	 &	 & 	 &Obj	&4.931	&4.931	&4.931	&4.931	&4.931	&4.931	&4.931\\ 
	 &	 & 	 &Gap(\%)	&0.00	&0.00	&0.00	&0.00	&0.00	&0.00	&0.00\\ 
	 &	 & 	 &\#Nodes	&1	&0	&1	&1	&0	& \multicolumn{1}{c}{---}	&0\\ 
	 &	 & 	 &\#Cuts	& \multicolumn{1}{c}{---}	&14	& \multicolumn{1}{c}{---}	&10	&4	&1	&7\\ \cline{3-11} 
 	 &	 &$10^5$	 &Time	&521.3	&3.8	&2126.2	&\textbf{3.5}	&110.3	&10.3	&23.1\\ 
	 &	 & 	 &Obj	&4.944	&4.944	&4.944	&4.944	&4.944	&4.944	&4.944\\ 
	 &	 & 	 &Gap(\%)	&0.00	&0.00	&0.00	&0.00	&0.00	&0.00	&0.00\\ 
	 &	 & 	 &\#Nodes	&1	&0	&1	&1	&0	& \multicolumn{1}{c}{---}	&0\\ 
	 &	 & 	 &\#Cuts	& \multicolumn{1}{c}{---}	&15	& \multicolumn{1}{c}{---}	&12	&4	&1	&7\\ \hline 
 \texttt{sbm100}	 &100	 &$10^3$	 &Time	&3.3	&1.1	&3.7	&4.2	&4.9	&\textbf{0.6}	&0.8\\ 
	 &	 & 	 &Obj	&4.337	&4.337	&4.337	&4.337	&4.337	&4.337	&4.337\\ 
	 &	 & 	 &Gap(\%)	&0.00	&0.00	&0.00	&0.00	&0.00	&0.00	&0.00\\ 
	 &	 & 	 &\#Nodes	&1	&396	&1	&4626	&14	& \multicolumn{1}{c}{---}	&0\\ 
	 &	 & 	 &\#Cuts	& \multicolumn{1}{c}{---}	&70	& \multicolumn{1}{c}{---}	&189	&21	&1	&7\\ \cline{3-11} 
 	 &	 &$10^4$	 &Time	&39.1	&8.1	&45.0	&9.3	&51.1	&\textbf{3.3}	&4.0\\ 
	 &	 & 	 &Obj	&4.397	&4.397	&4.397	&4.397	&4.397	&4.397	&4.397\\ 
	 &	 & 	 &Gap(\%)	&0.00	&0.00	&0.00	&0.00	&0.00	&0.00	&0.00\\ 
	 &	 & 	 &\#Nodes	&1	&1820	&1	&3274	&58	& \multicolumn{1}{c}{---}	&0\\ 
	 &	 & 	 &\#Cuts	& \multicolumn{1}{c}{---}	&103	& \multicolumn{1}{c}{---}	&107	&20	&1	&7\\ \cline{3-11} 
 	 &	 &$10^5$	 &Time	&1107.1	&65.1	&$>$3600	&84.6	&482.7	&\textbf{27.6}	&36.6\\ 
	 &	 & 	 &Obj	&4.364	&4.364	&4.364	&4.364	&4.364	&4.364	&4.364\\ 
	 &	 & 	 &Gap(\%)	&0.00	&0.00	&0.00	&0.00	&0.00	&0.00	&0.00\\ 
	 &	 & 	 &\#Nodes	&1	&412	&$>$1	&4676	&10	& \multicolumn{1}{c}{---}	&0\\ 
	 &	 & 	 &\#Cuts	& \multicolumn{1}{c}{---}	&120	& \multicolumn{1}{c}{---}	&152	&17	&1	&7\\  \bottomrule
	 \end{tabular}
\end{table}

\begin{table}[ht]
\scriptsize \renewcommand{\arraystretch}{0.9}
    \centering
    \caption{Numerical results for the datasets~\cite{Beasley1990} with $(\gamma,k)=(10/\sqrt{N},10)$}
    \label{tbl:results_scenario_port}
    \begin{threeparttable}
    \begin{tabular}{crrrrrrrrrr}
    \toprule
    \multirow{2}{*}{Data} 	&\multirow{2}{*}{$N$}	&\multirow{2}{*}{$S$}	 &&\multicolumn{2}{c}{BigM}	 &\multicolumn{2}{c}{Persp}	 &\multirow{2}{*}{CP}	 &\multirow{2}{*}{BCP}	  &\multirow{2}{*}{BCPc}\\ \cmidrule(lr){5-6} \cmidrule(lr){7-8} 
        &       &        &&\multicolumn{1}{c}{Lift}&\multicolumn{1}{c}{Cut} &\multicolumn{1}{c}{Lift}&\multicolumn{1}{c}{Cut}\\ \midrule
\texttt{port1}	 &31	 &$10^3$	 &Time	&1.1	&2.3	&1.5	&4.3	&3.0	&\textbf{0.8}	&1.7\\ 
	 &	 & 	 &Obj	&4.404	&4.404	&4.404	&4.404	&4.404	&4.404	&4.404\\ 
	 &	 & 	 &Gap(\%)	&0.00	&0.00	&0.00	&0.00	&0.00	&0.00	&0.00\\ 
	 &	 & 	 &\#Nodes	&1	&2301	&1	&2494	&58	& \multicolumn{1}{c}{---}	&0\\ 
	 &	 & 	 &\#Cuts	& \multicolumn{1}{c}{---}	&338	& \multicolumn{1}{c}{---}	&503	&36	&1	&7\\ \cline{3-11} 
 	 &	 &$10^4$	 &Time	&11.7	&11.4	&24.0	&17.7	&18.0	&\textbf{4.8}	&8.2\\ 
	 &	 & 	 &Obj	&4.374	&4.374	&4.374	&4.374	&4.374	&4.374	&4.374\\ 
	 &	 & 	 &Gap(\%)	&0.00	&0.00	&0.00	&0.00	&0.00	&0.00	&0.00\\ 
	 &	 & 	 &\#Nodes	&1	&850	&1	&3005	&46	& \multicolumn{1}{c}{---}	&0\\ 
	 &	 & 	 &\#Cuts	& \multicolumn{1}{c}{---}	&353	& \multicolumn{1}{c}{---}	&520	&17	&1	&8\\ \cline{3-11} 
 	 &	 &$10^5$	 &Time	&468.6	&163.3	&1933.2	&\textbf{15.8}	&265.8	&39.5	&64.6\\ 
	 &	 & 	 &Obj	&4.264	&4.264	&4.264	&4.264	&4.264	&4.264	&4.264\\ 
	 &	 & 	 &Gap(\%)	&0.00	&0.00	&0.00	&0.00	&0.00	&0.00	&0.00\\ 
	 &	 & 	 &\#Nodes	&1	&2433	&1	&253	&39	& \multicolumn{1}{c}{---}	&0\\ 
	 &	 & 	 &\#Cuts	& \multicolumn{1}{c}{---}	&654	& \multicolumn{1}{c}{---}	&56	&25	&1	&7\\ \hline 
 \texttt{port2}	 &89	 &$10^3$	 &Time	&\textbf{7.4}	&$>$3600	&7.5	&$>$3600	&$>$3600	&$>$3600	&$>$3600\\ 
	 &	 & 	 &Obj	&\textbf{1.773}	&2.119	&\textbf{1.773}	&2.002	&1.833 &2.022	&\textbf{1.773}\\ 
	 &	 & 	 &Gap(\%)	&\textbf{0.00}	&19.13	&\textbf{0.00}	&14.61	&8.52	&10.80	&6.90\\ 
	 &	 & 	 &\#Nodes	&1432	&$>$287,645	&577	&$>$152,730	&$>$216,174	& \multicolumn{1}{c}{---}&$>$36,644\\ 
	 &	 & 	 &\#Cuts	& \multicolumn{1}{c}{---}	&$>$17,731	& \multicolumn{1}{c}{---}	&$>$9784	&$>$15,007	&$>$7190	&$>$3899\\ \cline{3-11} 
 	 &	 &$10^4$	 &Time	&\textbf{311.2}	&$>$3600	&429.0	&$>$3600	&$>$3600	&$>$3600	&$>$3600\\ 
	 &	 & 	 &Obj	&\textbf{1.938}	&2.054	&\textbf{1.938}	&1.981	&2.009	&2.181	&1.943\\ 
	 &	 & 	 &Gap(\%)	&\textbf{0.00}	&13.17	&\textbf{0.00}	&6.90	&9.83	&10.61	&6.53\\ 
	 &	 & 	 &\#Nodes	&5715	&$>$261,628	&2429	&$>$204,649	&$>$22,111	& \multicolumn{1}{c}{---}&$>$18,077\\ 
	 &	 & 	 &\#Cuts	& \multicolumn{1}{c}{---}	&$>$14,278	& \multicolumn{1}{c}{---}	&$>$12,408	&$>$3654	&$>$2695	&$>$2548\\ \cline{3-11} 
 	 &	 &$10^5$	 &Time	&$>$3600	&$>$3600	&$>$3600	&$>$3600	&$>$3600	&$>$3600	&$>$3600\\ 
	 &	 & 	 &Obj	&2.030	&2.140	&2.039	&2.074	&1.988	&2.136	&\textbf{1.950}\\ 
	 &	 & 	 &Gap(\%)	&9.03	&18.61	&9.17	&17.76	&\textbf{8.35}	&9.30	&13.55\\ 
	 &	 & 	 &\#Nodes	&$>$481	&$>$99,640	&$>$1	&$>$40,749	&$>$1756	& \multicolumn{1}{c}{---} &$>$927\\ 
	 &	 & 	 &\#Cuts	& \multicolumn{1}{c}{---}	&$>$6168	& \multicolumn{1}{c}{---}	&$>$5123	&$>$406	&$>$350	&$>$285\\ \hline 
 \texttt{port5}	 &225	 &$10^3$	 &Time	&8.3	&250.0	&\textbf{7.9}	&40.6	&144.8	&157.8	&130.3\\ 
	 &	 & 	 &Obj	&\textbf{2.933}	&\textbf{2.933}	&\textbf{2.933}	&2.650\tnote{$\dagger$}	&\textbf{2.933}	&\textbf{2.933}	&\textbf{2.933}\\ 
	 &	 & 	 &Gap(\%)	&\textbf{0.00}	&\textbf{0.00}	&\textbf{0.00}	&0.00	&\textbf{0.00}	&\textbf{0.00}	&\textbf{0.00}\\ 
	 &	 & 	 &\#Nodes	&368	&45,800	&21	&4674	&7982	& \multicolumn{1}{c}{---}	&7499\\ 
	 &	 & 	 &\#Cuts	& \multicolumn{1}{c}{---}	&4632	& \multicolumn{1}{c}{---}	&608	&775	&446	&489\\ \cline{3-11} 
 	 &	 &$10^4$	 &Time	&246.1	&668.1	&\textbf{113.0}	&211.4	&1387.1	&500.1	&587.5\\ 
	 &	 & 	 &Obj	&\textbf{3.147}	&\textbf{3.147}	&\textbf{3.147}	&\textbf{3.147}	&\textbf{3.147}	&\textbf{3.147}	&\textbf{3.147}\\ 
	 &	 & 	 &Gap(\%)	&\textbf{0.00}	&\textbf{0.00}	&\textbf{0.00}	&\textbf{0.00}	&\textbf{0.00}	&\textbf{0.00}	&\textbf{0.00}\\ 
	 &	 & 	 &\#Nodes	&428	&42,363	&317	&9056	&8269	& \multicolumn{1}{c}{---}	&7347\\ 
	 &	 & 	 &\#Cuts	& \multicolumn{1}{c}{---}	&3773	& \multicolumn{1}{c}{---}	&1204	&701	&517	&727\\ \cline{3-11} 
 	 &	 &$10^5$	 &Time	&$>$3600	&$>$3600	&$>$3600	&$>$3600	&$>$3600	&\textbf{2351.6}	&3300.8\\ 
	 &	 & 	 &Obj	&$\infty$	&3.246	&3.143	&3.276	&3.173	&\textbf{3.138}	&\textbf{3.138}\\ 
	 &	 & 	 &Gap(\%)	&100.00	&3.98	&0.44	&4.29	&2.05	&\textbf{0.00}	&\textbf{0.00}\\ 
	 &	 & 	 &\#Nodes	&$>$0	&$>$20,833	&$>$61	&$>$16,790	&$>$2140	& \multicolumn{1}{c}{---}	&5697\\ 
	 &	 & 	 &\#Cuts	& \multicolumn{1}{c}{---}	&$>$2409	& \multicolumn{1}{c}{---}	&$>$2933	&$>$219	&352	&519\\ 
	 \bottomrule     
	 \end{tabular}
	 \begin{tablenotes}\footnotesize
    \item[$\dagger$] Gurobi returned an incorrect optimal objective value because of numerical instability.
    \end{tablenotes}
     \end{threeparttable}
\end{table}

\subsection{Sensitivity to hyperparameter values}
Next, we examine the sensitivity of the computational performance to the hyperparameters $\gamma$ for the $\ell_2$-regularization term and $k$ for the cardinality constraint. 
Here, we used three large datasets, namely, \texttt{ind49}, \texttt{sbm100}, and \texttt{port5}. 

Table \ref{tbl:results_gamma} presents the numerical results for each dataset with the $\ell_2$-regularization parameter $\gamma \in \{1/\sqrt{N}, 10/\sqrt{N}, 100/\sqrt{N}\}$.
Here, we set $k=10$ for the cardinality constraint and $S=10^5$ as the number of scenarios.
BCP was faster than the other methods except in the case of the \texttt{sbm100} dataset with $\gamma=100/\sqrt{N}$.
Additionally, BCP tended to be faster with smaller $\gamma$.
For the \texttt{port5} dataset, BCP attained an optimal solution in 562.3~s with $\gamma=1/\sqrt{N}$, whereas it reached the time limit with $\gamma=100/\sqrt{N}$, resulting in a solution of poor quality.
In contrast, BCPc found a good solution even when the computation was terminated because of the time limit.

Table \ref{tbl:results_k} gives the numerical results for each dataset with the cardinality parameter $k \in \{5, 10, 15\}$.
Here, we set $\gamma=10/\sqrt{N}$ for the $\ell_2$-regularization term and $S=10^5$ as the number of scenarios.
BCP outperformed the other methods for most of the problem instances and tended to be faster with larger $k$. 
For the \texttt{port5} dataset, BCP solved the problem with $k=15$ in 73.7~s, whereas it failed to finish the computation with $k=5$.
Here, \#Cuts was also much smaller for $k=15$ than for $k=5$, which is part of the reason why BCP performed better with larger $k$. 
The results of BCPc show a similar tendency.

\begin{table}[ht]
\scriptsize \renewcommand{\arraystretch}{0.9}
    \centering
    \caption{Sensitivity results for large problem instances with $(S,k)=(10^5,10)$}
    \label{tbl:results_gamma}
    \begin{tabular}{crrrrrrrrrrr}
    \toprule
    \multirow{2}{*}{Data} 	&\multirow{2}{*}{$N$}	&\multirow{2}{*}{$\gamma$}	 &&\multicolumn{2}{c}{BigM}	 &\multicolumn{2}{c}{Persp}	 &\multirow{2}{*}{CP}	 &\multirow{2}{*}{BCP}	  &\multirow{2}{*}{BCPc}\\  \cmidrule(lr){5-6} \cmidrule(lr){7-8} 
        &       &        &&\multicolumn{1}{c}{Lift}&\multicolumn{1}{c}{Cut} &\multicolumn{1}{c}{Lift}&\multicolumn{1}{c}{Cut}\\ \midrule
\texttt{ind49}	 &49	 &1$/\sqrt{N}$	 &Time	&1138.4	&246.5	&$>$3600	&209.7	&823.3	&\textbf{64.9}	&127.6\\ 
	 &	 & 	 &Obj	&3.823	&3.823	&3.824	&3.823	&3.823	&3.823	&3.823\\ 
	 &	 & 	 &Gap(\%)	&0.00	&0.00	&0.12	&0.00	&0.00	&0.00	&0.00\\ 
	 &	 & 	 &\#Nodes	&483	&4560	&$>$6	&4647	&656	& \multicolumn{1}{c}{---}	&32\\ 
	 &	 & 	 &\#Cuts	& \multicolumn{1}{c}{---}	&696	& \multicolumn{1}{c}{---}	&517	&93	&14	&37\\ \cline{3-11} 
 	 &	 &10$/\sqrt{N}$	 &Time	&709.8	&125.0	&$>$3600	&220.0	&662.2	&\textbf{43.1}	&103.6\\ 
	 &	 & 	 &Obj	&3.379	&3.379	&3.380	&3.379	&3.379	&3.379	&3.379\\ 
	 &	 & 	 &Gap(\%)	&0.00	&0.00	&0.01	&0.00	&0.00	&0.00	&0.00\\ 
	 &	 & 	 &\#Nodes	&31	&4101	&$>$1	&3489	&$>$251	& \multicolumn{1}{c}{---}	&0\\ 
	 &	 & 	 &\#Cuts	& \multicolumn{1}{c}{---}	&380	& \multicolumn{1}{c}{---}	&$>$661	&68	&1	&12\\ \cline{3-11} 
 	 &	 &100$/\sqrt{N}$	 &Time	&528.2	&415.4	&3016.6	&556.1	&2264.5	&\textbf{89.0}	&129.2\\ 
	 &	 & 	 &Obj	&3.322	&3.322	&3.322	&3.322	&3.322	&3.322	&3.322\\ 
	 &	 & 	 &Gap(\%)	&0.00	&0.00	&0.00	&0.00	&0.00	&0.00	&0.00\\ 
	 &	 & 	 &\#Nodes	&1	&5273	&1	&4464	&2391	& \multicolumn{1}{c}{---}	&0\\ 
	 &	 & 	 &\#Cuts	& \multicolumn{1}{c}{---}	&1172	& \multicolumn{1}{c}{---}	&1544	&290	&1	&7\\ \hline 
 \texttt{sbm100}	 &100	 &1$/\sqrt{N}$	 &Time	&$>$3600	&114.8	&$>$3600	&60.8	&495.8	&\textbf{22.3}	&34.7\\ 
	 &	 & 	 &Obj	&5.446	&5.075	&5.075	&5.075	&5.075	&5.075	&5.075\\ 
	 &	 & 	 &Gap(\%)	&7.19	&0.00	&0.01	&0.00	&0.00	&0.00	&0.00\\ 
	 &	 & 	 &\#Nodes	&$>$1	&2741	&$>$5	&1791	&21	& \multicolumn{1}{c}{---}	&1\\ 
	 &	 & 	 &\#Cuts	& \multicolumn{1}{c}{---}	&175	& \multicolumn{1}{c}{---}	&80	&21	&3	&12\\ \cline{3-11} 
 	 &	 &10$/\sqrt{N}$	 &Time	&1107.1	&65.1	&$>$3600	&84.6	&482.7	&\textbf{27.6}	&36.6\\ 
	 &	 & 	 &Obj	&4.364	&4.364	&4.364	&4.364	&4.364	&4.364	&4.364\\ 
	 &	 & 	 &Gap(\%)	&0.00	&0.00	&0.00	&0.00	&0.00	&0.00	&0.00\\ 
	 &	 & 	 &\#Nodes	&1	&412	&$>$1	&4676	&10	& \multicolumn{1}{c}{---}	&0\\ 
	 &	 & 	 &\#Cuts	& \multicolumn{1}{c}{---}	&120	& \multicolumn{1}{c}{---}	&152	&17	&1	&7\\ \cline{3-11} 
 	 &	 &100$/\sqrt{N}$	 &Time	&1078.9	&155.1	&1977.1	&\textbf{22.8}	&557.6	&46.8	&74.6\\ 
	 &	 & 	 &Obj	&4.225	&4.225	&4.225	&4.225	&4.225	&4.225	&4.225\\ 
	 &	 & 	 &Gap(\%)	&0.00	&0.00	&0.00	&0.00	&0.00	&0.00	&0.00\\ 
	 &	 & 	 &\#Nodes	&1	&634	&1	&1	&19	& \multicolumn{1}{c}{---}	&0\\ 
	 &	 & 	 &\#Cuts	& \multicolumn{1}{c}{---}	&252	& \multicolumn{1}{c}{---}	&40	&27	&1	&9\\ \hline 
 \texttt{port5}	 &225	 &1$/\sqrt{N}$	 &Time	&$>$3600	&$>$3600	&$>$3600	&3392.3	&$>$3600	&\textbf{562.3}	&873.9\\ 
	 &	 & 	 &Obj	&$\infty$	&3.945	&3.899	&\textbf{3.841}	&3.877	&\textbf{3.841}	&\textbf{3.841}\\ 
	 &	 & 	 &Gap(\%)	&100.00	&9.96	&4.97	&\textbf{0.00}	&2.69	&\textbf{0.00}	&\textbf{0.00}\\ 
	 &	 & 	 &\#Nodes	&$>$0	&$>$184,221	&$>$1	&20,182	&$>$2311	& \multicolumn{1}{c}{---}	&3124\\ 
	 &	 & 	 &\#Cuts	& \multicolumn{1}{c}{---}	&$>$2610	& \multicolumn{1}{c}{---}	&2125	&$>$203	&167	&275\\ \cline{3-11} 
 	 &	 &10$/\sqrt{N}$	 &Time	&$>$3600	&$>$3600	&$>$3600	&$>$3600	&$>$3600	&\textbf{2351.6}	&3300.8\\ 
	 &	 & 	 &Obj	&$\infty$	&3.246	&3.143	&3.276	&3.173	&\textbf{3.138}	&\textbf{3.138}\\ 
	 &	 & 	 &Gap(\%)	&100.00	&3.98	&0.44	&4.29	&2.05	&\textbf{0.00}	&\textbf{0.00}\\ 
	 &	 & 	 &\#Nodes	&$>$0	&$>$20,833	&$>$61	&$>$16,790	&$>$2140	& \multicolumn{1}{c}{---}	&5697\\ 
	 &	 & 	 &\#Cuts	& \multicolumn{1}{c}{---}	&$>$2409	& \multicolumn{1}{c}{---}	&$>$2933	&$>$219	&352	&519\\ \cline{3-11} 
 	 &	 &100$/\sqrt{N}$	 &Time	&$>$3600	&$>$3600	&$>$3600	&$>$3600	&$>$3600	&$>$3600	&$>$3600\\ 
	 &	 & 	 &Obj	&$\infty$	&$\infty$&\textbf{3.058}	&3.109	&3.868	&3.454	&3.061\\ 
	 &	 & 	 &Gap(\%)	&100.00	&100.00	&\textbf{0.30}	&1.79	&21.44	&9.12	&15.18\\ 
	 &	 & 	 &\#Nodes	&$>$0	&$>$11,144	&$>$31	&$>$14,020	&$>$826	& \multicolumn{1}{c}{---} &$>$547\\ 
	 &	 & 	 &\#Cuts	& \multicolumn{1}{c}{---}	&$>$2738	& \multicolumn{1}{c}{---}	&$>$2737	&$>$241	&$>$283	&$>$223\\ 
 \bottomrule     
	 \end{tabular}
\end{table}

\begin{table}[ht]
\scriptsize \renewcommand{\arraystretch}{0.9}
    \centering
    \caption{Sensitivity results for large problem instances with $(S,\gamma)=(10^5,10/\sqrt{N})$}
    \label{tbl:results_k}
    \begin{tabular}{crrrrrrrrrrr}
    \toprule
    \multirow{2}{*}{Data} 	&\multirow{2}{*}{$N$}	&\multirow{2}{*}{$k$}	 &&\multicolumn{2}{c}{BigM}	 &\multicolumn{2}{c}{Persp}	 &\multirow{2}{*}{CP}	 &\multirow{2}{*}{BCP}	  &\multirow{2}{*}{BCPc}\\ \cmidrule(lr){5-6} \cmidrule(lr){7-8} 
        &       &        &&\multicolumn{1}{c}{Lift}&\multicolumn{1}{c}{Cut} &\multicolumn{1}{c}{Lift}&\multicolumn{1}{c}{Cut}\\ \midrule 
\texttt{ind49}	 &49	 &5	 &Time	&733.7	&\textbf{160.0}	&$>$3600	&434.3	&1455.9	&638.5	&907.5\\ 
	 &	 & 	 &Obj	&3.443	&3.443	&3.443	&3.443	&3.443	&3.443	&3.443\\ 
	 &	 & 	 &Gap(\%)	&0.00	&0.00	&0.00	&0.00	&0.00	&0.00	&0.00\\ 
	 &	 & 	 &\#Nodes	&83	&2475	&$>$61	&5172	&2614	& \multicolumn{1}{c}{---}	&970\\ 
	 &	 & 	 &\#Cuts	& \multicolumn{1}{c}{---}	&428	& \multicolumn{1}{c}{---}	&1197	&226	&132	&172\\ \cline{3-11} 
 	 &	 &10	 &Time	&709.8	&125.0	&$>$3600	&220.0	&662.2	&\textbf{43.1}	&103.6\\ 
	 &	 & 	 &Obj	&3.379	&3.379	&3.380	&3.379	&3.379	&3.379	&3.379\\ 
	 &	 & 	 &Gap(\%)	&0.00	&0.00	&0.00	&0.00	&0.00	&0.00	&0.00\\ 
	 &	 & 	 &\#Nodes	&31	&4101	&$>$1	&3489	&250	& \multicolumn{1}{c}{---}	&0\\ 
	 &	 & 	 &\#Cuts	& \multicolumn{1}{c}{---}	&380	& \multicolumn{1}{c}{---}	&661	&68	&1	&12\\ \cline{3-11} 
 	 &	 &15	 &Time	&555.2	&128.2	&$>$3600	&76.7	&448.6	&\textbf{48.7}	&82.3\\ 
	 &	 & 	 &Obj	&3.366	&3.366	&3.366	&3.366	&3.366	&3.366	&3.366\\ 
	 &	 & 	 &Gap(\%)	&0.00	&0.00	&0.00	&0.00	&0.00	&0.00	&0.00\\ 
	 &	 & 	 &\#Nodes	&1	&3392	&$>$1	&1043	&67	& \multicolumn{1}{c}{---}	&0\\ 
	 &	 & 	 &\#Cuts	& \multicolumn{1}{c}{---}	&363	& \multicolumn{1}{c}{---}	&211	&39	&1	&7\\ \hline 
 \texttt{sbm100}	 &100	 &5	 &Time	&1260.7	&179.8	&2941.5	&147.8	&690.8	&\textbf{38.3}	&48.7\\ 
	 &	 & 	 &Obj	&4.367	&4.367	&4.367	&4.367	&4.367	&4.367	&4.367\\ 
	 &	 & 	 &Gap(\%)	&0.00	&0.00	&0.00	&0.00	&0.00	&0.00	&0.00\\ 
	 &	 & 	 &\#Nodes	&92	&1092	&3	&1866	&116	& \multicolumn{1}{c}{---}	&7\\ 
	 &	 & 	 &\#Cuts	& \multicolumn{1}{c}{---}	&296	& \multicolumn{1}{c}{---}	&239	&53	&5	&12\\ \cline{3-11} 
 	 &	 &10	 &Time	&1107.1	&65.1	&$>$3600	&84.6	&482.7	&\textbf{27.6}	&36.6\\ 
	 &	 & 	 &Obj	&4.364	&4.364	&4.364	&4.364	&4.364	&4.364	&4.364\\ 
	 &	 & 	 &Gap(\%)	&0.00	&0.00	&0.00	&0.00	&0.00	&0.00	&0.00\\ 
	 &	 & 	 &\#Nodes	&1	&412	&$>$1	&4676	&10	& \multicolumn{1}{c}{---}	&0\\ 
	 &	 & 	 &\#Cuts	& \multicolumn{1}{c}{---}	&120	& \multicolumn{1}{c}{---}	&152	&17	&1	&7\\ \cline{3-11} 
 	 &	 &15	 &Time	&1051.6	&43.2	&3438.1	&44.4	&437.5	&\textbf{29.5}	&48.0\\ 
	 &	 & 	 &Obj	&4.364	&4.364	&4.364	&4.364	&4.364	&4.364	&4.364\\ 
	 &	 & 	 &Gap(\%)	&0.00	&0.00	&0.00	&0.00	&0.00	&0.00	&0.00\\ 
	 &	 & 	 &\#Nodes	&1	&90	&1	&5738	&9	& \multicolumn{1}{c}{---}	&0\\ 
	 &	 & 	 &\#Cuts	& \multicolumn{1}{c}{---}	&70	& \multicolumn{1}{c}{---}	&64	&12	&1	&7\\ \hline 
 \texttt{port5}	 &225	 &5	 &Time	&$>$3600	&$>$3600	&$>$3600	&\textbf{2263.9}	&$>$3600	&$>$3600	&$>$3600\\ 
	 &	 & 	 &Obj	&$\infty$ &3.491	&3.755	&\textbf{3.295}	&3.397	&\textbf{3.295}	&\textbf{3.295}\\ 
	 &	 & 	 &Gap(\%)	&100.00	&8.72	&14.25	&\textbf{0.00}	&8.52	&0.72	&3.37\\ 
	 &	 & 	 &\#Nodes	&$>$0	&$>$41,433	&$>$42	&19,851	&$>$1390	& \multicolumn{1}{c}{---}	&$>$10,767\\ 
	 &	 & 	 &\#Cuts	& \multicolumn{1}{c}{---}	&$>$2644	& \multicolumn{1}{c}{---}	&1649	&$>$249	&$>$667	&$>$861\\ \cline{3-11} 
 	 &	 &10	 &Time	&$>$3600	&$>$3600	&$>$3600	&$>$3600	&$>$3600	&\textbf{2351.6}	&3300.8\\ 
	 &	 & 	 &Obj	&$\infty$	&3.246	&3.143	&3.276	&3.173	&\textbf{3.138}	&\textbf{3.138}\\ 
	 &	 & 	 &Gap(\%)	&100.00	&3.98	&0.44	&4.29	&2.05	&\textbf{0.00}	&\textbf{0.00}\\ 
	 &	 & 	 &\#Nodes	&$>$0	&$>$20,833	&$>$61	&$>$16,790	&$>$2140	& \multicolumn{1}{c}{---}	&5697\\ 
	 &	 & 	 &\#Cuts	& \multicolumn{1}{c}{---}	&$>$2409	& \multicolumn{1}{c}{---}	&$>$2933	&$>$219	&352	&519\\ \cline{3-11} 
 	 &	 &15	 &Time	&$>$3600	&1154.4	&$>$3600	&1116.1	&$>$3600	&\textbf{73.7}	&121.4\\ 
	 &	 & 	 &Obj	&$\infty$	&\textbf{3.109}	&\textbf{3.109}	&\textbf{3.109}	&3.170	&\textbf{3.109}	&\textbf{3.109}\\ 
	 &	 & 	 &Gap(\%)	&100.00	&\textbf{0.00}	&0.01	&\textbf{0.00}	&1.97	&\textbf{0.00}	&\textbf{0.00}\\ 
	 &	 & 	 &\#Nodes	&$>$0	&8454	&$>$1	&5696	&$>$2516	& \multicolumn{1}{c}{---}	&0\\ 
	 &	 & 	 &\#Cuts	& \multicolumn{1}{c}{---}	&757	& \multicolumn{1}{c}{---}	&838	&$>$218	&1	&11\\
 \bottomrule     
	 \end{tabular}
\end{table}

\section{Conclusion}\label{sec:conclusion}
This paper deals with the scenario-based mean-CVaR portfolio optimization problem with the cardinality constraint.
It is very hard to exactly solve the problem when it involves a large number of investable assets. 
In addition, accurate approximation of scenario-based CVaR requires sufficiently many scenarios, which decreases computational efficiency. 
We reformulated the problem as a bilevel optimization problem and developed a cutting-plane algorithm specialized for solving the upper-level problem. 
Moreover, we integrated another cutting-plane algorithm to efficiently solve the lower-level problem with a large number of scenarios.  
We also proved that our algorithms give a solution with guaranteed global optimality in a finite number of iterations. 

The computational results indicate that our cutting-plane algorithms were very effective especially when the number of scenarios was large. 
Remarkably, our bilevel cutting-plane algorithm attained an optimal solution within 3600~s to a problem involving 225 assets and 100,000 scenarios.
Furthermore, our algorithms performed well for most of the hyperparameter values.  

A future direction of study is to extend our algorithm to mixed-integer conic optimization problems related to robust portfolio optimization models~\cite{fabozzi2010robust}. 
For such problems, some previous studies have pointed out the effectiveness of cutting-plane algorithms~\cite{bertsimas2019unified,Coey2020,Kobayashi2019}. 
Another direction of future research is to use statistical techniques~\cite{bertsimas2020scalable,Tamura2017,tamura2019mixed} to increase the numerical stability of computed portfolios.

\bibliographystyle{spmpsci}      


\bibliographystyle{siam}
\bibliography{bibliography} 

\appendix
\def\thesection{Appendix \Alph{section}}

\section{Proof of Theorem \ref{thm:f_dual_lifting}}\label{sec:proof_lift}

Problem \eqref{prob:sub2} is formulated as follows:
 \begin{subequations}\label{prob:sub_reform} 
\begin{alignat}{3}
f(\bm z)=&\minimize_{a,\bm q,v,\bm x} &&\quad \frac{1}{2\gamma }\bm x^\top \bm x + a +  v\\
    &\subjectto &&\quad  v\geq \frac{1}{1-\beta}\sum_{s \in \mathcal{S}}p_s q_s, \label{sub_const1}\\
    &&&\quad q_s \geq -(\bm r^{(s)})^\top \bm Z\bm x - a &\quad (\forall s\in \mathcal{S}),\label{sub_const2}\\
    &&&\quad \bm A \bm Z \bm x \leq \bm b,\label{sub_const3}\\
    &&&\quad \bm 1^\top \bm Z \bm x = 1,\label{sub_const4}\\
    &&&\quad \bm Z \bm x \geq \bm 0,\label{sub_const5}\\
    &&&\quad  \bm q\geq \bm 0. \label{sub_const6}
\end{alignat}
\end{subequations}
The Lagrange function of Problem~\eqref{prob:sub_reform} is expressed as
\begin{equation*}
    \begin{split}
        \mathcal{L}(a,\bm q,v,\bm x, \eta,\bm \alpha, \bm \zeta, \lambda,\bm \pi,\bm \theta) 
        & := \frac{1}{2\gamma }\bm x^\top \bm x + a +  v - \eta\left(v-\frac{1}{1-\beta}\sum_{s \in \mathcal{S}}p_sq_s\right)\\
        &\qquad -\sum_{ s\in \mathcal{S}}\alpha_{s}\left(q_s +(\bm r^{(s)})^\top\bm Z \bm x +a\right)\\
        &\qquad  -\bm \zeta^\top \left(\bm b - \bm A\bm Z\bm x\right) - \lambda \left( \bm 1^\top \bm Z\bm x-1\right)-\bm \pi^\top \bm  Z\bm x- \bm \theta^\top \bm q,
    \end{split}
\end{equation*}
where $\eta \geq 0,~\bm \alpha := (\alpha_{s})_{s\in \mathcal{S}} \geq \bm 0,~\bm \zeta \geq \bm 0,~\lambda\in \mathbb{R},~\bm \pi \geq \bm 0~\mathrm{and}~\bm \theta \geq \bm 0$ are Lagrange multipliers. 
Then, the Lagrange dual problem of Problem \eqref{prob:sub_reform} is posed as:
 \begin{equation}\label{prob:sub_relax_dual}
    \max_{\substack{\eta \geq 0,~\bm \alpha\geq \bm 0,~\bm \zeta\geq \bm 0,\\ \lambda\in \mathbb{R},~\bm \pi \geq \bm 0,~\bm \theta \geq \bm 0}}~\min_{\substack{a\in \mathbb{R},~\bm q \in \mathbb{R}^S,\\ v\in \mathbb{R},~\bm x\in \mathbb{R}^N}}~\mathcal{L}(a,\bm q,v,\bm x, \eta,\bm \alpha, \bm \zeta, \lambda,\bm \pi,\bm \theta).
\end{equation}

Recall that Problem~\eqref{prob:sub_reform} is feasible.
Also, the objective function is proper convex, and all the constraints are linear in Problem~\eqref{prob:sub_reform}.
Then, the strong duality holds; see, for example, Section~5.2.3 in Boyd and Vandenberghe~\cite{Boyd2009}.
As a result, $f(\bm z)$ is equal to the optimal objective value of Problem \eqref{prob:sub_relax_dual}.
Now, let us focus on the inner minimization problem: 
\begin{equation}\label{prob:relaxed_prob}
    \min_{\substack{a\in \mathbb{R},~\bm q \in \mathbb{R}^S,\\ v\in \mathbb{R},~\bm x\in \mathbb{R}^N}}~\mathcal{L}(a,\bm q,v,\bm x, \eta,\bm \alpha, \bm \zeta, \lambda,\bm \pi,\bm \theta).
\end{equation}
Note that Problem \eqref{prob:relaxed_prob} is an unconstrained convex quadratic optimization problem and its objective function is linear in $(a, \bm q, v)$.
Because Problem (\ref{prob:relaxed_prob}) must be bounded, the Lagrange multipliers are required to satisfy the following conditions:
\begin{align}
    \nabla_{a}\mathcal{L} &= 1 - \sum_{s\in \mathcal{S}}\alpha_s = 0, \label{dual:codition1}\\
    \nabla_{\bm q}\mathcal{L} &= \frac{\eta}{1-\beta}\bm p - \bm \alpha-\bm \theta = 0,\label{dual:codition2}\\
    \nabla_{v}\mathcal{L} &=1 -\eta  = 0.\label{dual:codition3}
    \end{align}
Also, the following optimality condition should be satisfied:
\begin{align}
        \nabla_{\bm x}\mathcal{L} &= \frac{1}{\gamma} \bm x - \bm Z \left(\sum_{s\in \mathcal{S}}\alpha_{s}\bm r^{(s)} - \bm A^{\top} \bm \zeta + \lambda \bm 1 + \bm \pi \right) = \bm 0. \label{dual:codition4}
\end{align}

According to Eqs. (\ref{dual:codition1}), (\ref{dual:codition2}), (\ref{dual:codition3}), and (\ref{dual:codition4}), the optimal objective value of Problem (\ref{prob:relaxed_prob}) is calculated as
\begin{equation*}
 -\frac{\gamma}{2} \bm \omega^\top \bm Z^2 \bm \omega - \bm b^\top \bm \zeta + \lambda, 
\end{equation*}
where $\bm \omega\in \mathbb{R}^N$ is a vector of auxiliary decision variables satisfying
\begin{equation*}
    \bm \omega = \sum_{s\in \mathcal{S}}\alpha_{s}\bm r^{(s)} - \bm A^\top \bm \zeta + \lambda \bm 1 + \bm \pi.
\end{equation*}

Because $\bm z \in \mathcal{Z}_N^k$, it holds that $\bm \omega^\top \bm Z^2\bm \omega = \bm z^\top (\bm \omega\circ \bm \omega)$.
Therefore, the Lagrange dual problem~\eqref{prob:sub_relax_dual} is formulated as follows:
\begin{alignat*}{3}
    f(\bm z)=&\maximize_{\bm \alpha, \bm \zeta, \lambda, \bm \omega} &&\quad -\frac{\gamma}{2} \bm z^\top (\bm \omega\circ \bm \omega) - \bm b^\top \bm \zeta + \lambda\\
    &\subjectto && \quad \bm \omega \geq  \sum_{s\in \mathcal{S}}\alpha_{s}\bm r^{(s)} - \bm A^\top \bm \zeta + \lambda \bm 1, \\
    &&&\quad \sum_{s\in  \mathcal{S}}\alpha_s = 1,\\
    &&&\quad \alpha_s \leq\frac{p_s}{1-\beta} \quad(\forall s\in \mathcal{S}),\\
    &&&\quad \bm \alpha \geq \bm 0,~\bm \zeta\geq \bm 0,
\end{alignat*}
where we substitute $\eta=1$, and eliminate the nonnegative variables $\bm \pi$ and $\bm \theta$.
\qed

\section{Proof of Theorem \ref{thm:f_dual_cutting}}\label{sec:proof_cutting}
Problem \eqref{prob:sub_primal_part} is formulated as follows:
 \begin{subequations}\label{prob:sub_primal_part_reform} 
\begin{alignat}{3}
f_{\mathcal{K}}(\bm z)=&\minimize_{a, v,\bm x} &&\quad \frac{1}{2\gamma }\bm x^\top \bm x + a +  v\\
    &\subjectto &&\quad  v \geq \frac{1}{1-\beta}\sum_{s \in \mathcal{J}}p_s( -(\bm r^{(s)})^\top \bm Z\bm x - a) &&\quad (\forall \mathcal{J} \in \mathcal{K}),\\
    &&&\quad v\geq 0,\\
    &&&\quad \bm A \bm Z \bm x \leq \bm b,\\
    &&&\quad \bm 1^\top \bm Z \bm x = 1,\\
    &&&\quad \bm Z \bm x \geq \bm 0.
\end{alignat}
\end{subequations}
The Lagrange function of Problem~\eqref{prob:sub_primal_part_reform} is expressed as
\begin{equation*}
    \begin{split}
        \mathcal{L}(a, v, \bm x, \bm\alpha,\xi,\bm \zeta,\lambda,\bm \pi)&:= \frac{1}{2\gamma }\bm x^\top \bm x + a +  v-\sum_{ \mathcal{J}\in \mathcal{K}}\alpha_{\mathcal{J}}\left(v + \frac{1}{1-\beta}\sum_{s\in \mathcal{J}}p_s\left((\bm r^{(s)})^\top\bm Z \bm x +a\right)\right)\\
        &\qquad-\xi v-\bm \zeta^\top \left(\bm b - \bm A\bm Z\bm x\right) - \lambda \left( \bm 1^\top \bm Z\bm x-1\right)-\bm \pi^\top \bm Z\bm x,
    \end{split}
\end{equation*}
where $\bm \alpha := (\alpha_{\mathcal{J}})_{\mathcal{J}\in \mathcal{K}} \geq \bm 0,~\xi\ge 0,~ \bm \zeta\geq \bm 0,~\lambda\in \mathbb{R}~\mathrm{and}~\bm \pi \geq \bm 0$ are Lagrange multipliers.
Recall that Problem~\eqref{prob:sub_primal_part_reform} is feasible.
Then, the strong duality holds as well as the proof of Theorem \ref{thm:f_dual_lifting}, and $f_{\mathcal{K}}(\bm z)$ is equal to the optimal objective value of the following Lagrange dual problem:
\begin{equation}\label{prob:sub_relax_dual_cutting}
    \max_{\substack{\bm \alpha\geq \bm 0,~\xi\geq 0,~\bm \zeta\geq \bm 0,\\ \lambda\in \mathbb{R},~\bm \pi \geq \bm 0}}~\min_{a\in \mathbb{R},~v\in \mathbb{R},~\bm x\in \mathbb{R}^N} \quad \mathcal{L}(a, v, \bm x, \bm\alpha,    \xi,\bm \zeta,\lambda,\bm \pi).
\end{equation}

Now, let us consider the inner minimization problem:
\begin{equation}\label{prob:relaxed_prob_cutting}
    \min_{a\in \mathbb{R},~v\in \mathbb{R},~\bm x\in \mathbb{R}^N} \quad \mathcal{L}(a, v, \bm x, \bm\alpha,    \xi,\bm \zeta,\lambda,\bm \pi).
\end{equation}
Similarly to the proof of Theorem \ref{thm:f_dual_lifting}, the following conditions must be satisfied:
\begin{align}
    & \nabla_{a}\mathcal{L} = 1 - \frac{1}{1-\beta}\sum_{\mathcal{J}\in \mathcal{K}}\alpha_{\mathcal{J}} \sum_{s\in \mathcal J} p_s = 0, \label{dual:codition1_cutting}\\
    & \nabla_{v}\mathcal{L} =1 -\sum_{\mathcal{J}\in \mathcal{K}}\alpha_{\mathcal{J}}-\xi  = 0, \label{dual:codition2_cutting}\\
    & \nabla_{\bm x}\mathcal{L}=  \frac{1}{\gamma} \bm x - \bm Z \left(\frac{1}{1-\beta}\sum_{\mathcal J\in \mathcal{K}}\alpha_{\mathcal{J}} \sum_{s\in \mathcal{J}} p_s \bm r^{(s)} - \bm A^\top \bm \zeta  + \lambda \bm 1 + \bm \pi \right) = \bm 0. \label{dual:codition3_cutting}
\end{align}
According to Eqs. \eqref{dual:codition1_cutting}, \eqref{dual:codition2_cutting} and \eqref{dual:codition3_cutting}, the optimal objective value of Problem \eqref{prob:relaxed_prob_cutting} is calculated as
\begin{equation*}
 -\frac{\gamma}{2} \bm \omega^\top \bm Z^2 \bm \omega - \bm b^\top \bm \zeta + \lambda, 
\end{equation*}
where $\bm \omega\in \mathbb{R}^N$ is a vector of auxiliary decision variables satisfying
\begin{equation*}
    \bm \omega  = \frac{1}{1-\beta}\sum_{\mathcal J\in \mathcal{K}}\alpha_{\mathcal{J}} \sum_{s\in \mathcal{J}} p_s \bm r^{(s)} - \bm A^\top \bm\zeta  + \lambda \bm 1 + \bm \pi.
\end{equation*}

Because $\bm z \in \mathcal{Z}_N^k$, it holds that $\bm \omega^\top \bm Z^2\bm \omega = \bm z^\top (\bm \omega\circ \bm \omega)$. 
Therefore, the Lagrange dual problem \eqref{prob:sub_relax_dual_cutting} is formulated as follows:
\begin{alignat*}{3}
    f_{\mathcal{K}}(\bm z)=&\maximize_{\bm \alpha, \bm \zeta, \lambda, \bm \omega} &&\quad -\frac{\gamma}{2} \bm z^\top (\bm \omega\circ \bm \omega) - \bm b^\top \bm \zeta + \lambda\\
    &\subjectto && \quad \bm \omega \geq \frac{1}{1-\beta}\sum_{\mathcal J \in\mathcal{K}} \alpha_{\mathcal{J}} \sum_{s\in \mathcal J}p_s \bm r^{(s)} -\bm A^\top \bm \zeta +\lambda \bm 1, \label{eq:comp_omega} \\
    &&&\quad \sum_{\mathcal{J}\in \mathcal{K}}\alpha_{\mathcal{J}} \leq 1,\\
    &&&\quad \sum_{\mathcal{J}\in \mathcal{K}}\alpha_{\mathcal{J}} \sum_{s\in \mathcal{J}}p_s = 1-\beta,\\
    &&&\quad \bm \alpha \geq \bm 0,~\bm \zeta\geq \bm 0,
\end{alignat*}
where we eliminate the nonnegative variables $\xi$ and $\bm \pi$.
\qed 

\end{document}